\documentclass[a4paper]{amsart}
\pdfoutput=1
\usepackage[utf8]{inputenc}
\usepackage[T1]{fontenc}
\usepackage{lmodern}
\usepackage{amssymb}
\usepackage{nicefrac,mathtools,enumitem,amsmath,amsxtra}
\usepackage{microtype}

\usepackage{tikzexternal}

\tikzset{dar/.style={double,double equal sign distance,-implies}}
\tikzset{mid/.style={anchor=mid}} 

\tikzset{triar/.style={anchor=mid,->}}
\tikzset{tridar/.style={anchor=mid,double,double equal sign distance,-implies}}
\tikzset{narrowfill/.style={inner sep=0pt, fill=white}}

\setlist[enumerate,1]{label=\textup{(\arabic*)}}
\setlist[enumerate,2]{label=\textup{(\alph*)}}

\usepackage[pdftitle={Noncommutative coarse geometry},
pdfauthor={Tathagata Banerjee and Ralf Meyer},
pdfsubject={Mathematics}
]{hyperref}
\usepackage[lite]{amsrefs}

\renewcommand{\PrintDOI}[1]{\href{http://dx.doi.org/\detokenize{#1}}{doi: \detokenize{#1}}}

\BibSpec{book}{%
  +{}  {\PrintPrimary}                {transition}
  +{,} { \textit}                     {title}
  +{.} { }                            {part}
  +{:} { \textit}                     {subtitle}
  +{,} { \PrintEdition}               {edition}
  +{}  { \PrintEditorsB}              {editor}
  +{,} { \PrintTranslatorsC}          {translator}
  +{,} { \PrintContributions}         {contribution}
  +{,} { }                            {series}
  +{,} { \voltext}                    {volume}
  +{,} { }                            {publisher}
  +{,} { }                            {organization}
  +{,} { }                            {address}
  +{,} { \PrintDateB}                 {date}
  +{,} { }                            {status}
  +{}  { \parenthesize}               {language}
  +{}  { \PrintTranslation}           {translation}
  +{;} { \PrintReprint}               {reprint}
  +{.} { }                            {note}
  +{.} {}                             {transition}
  +{} { \PrintDOI}                   {doi}
  +{} { available at \url}            {eprint}
  +{}  {\SentenceSpace \PrintReviews} {review}
}

\theoremstyle{plain}
\newtheorem{theorem}[subsection]{Theorem}
\newtheorem{lemma}[subsection]{Lemma}
\newtheorem{proposition}[subsection]{Proposition}

\newtheorem{corollary}[subsection]{Corollary}

\theoremstyle{definition}
\newtheorem{definition}[subsection]{Definition}

\theoremstyle{remark}
\newtheorem{remark}[subsection]{Remark}

\newtheorem{example}[subsection]{Example}

\newcommand*{\defeq}{\mathrel{\vcentcolon=}}

\newcommand*{\braket}[2]{\langle#1, #2\rangle}
\newcommand*{\congto}{\xrightarrow\sim}

\newcommand{\idealin}{\mathrel{\triangleleft}} 
\newcommand*{\into}{\rightarrowtail}
\newcommand*{\prto}{\twoheadrightarrow}

\newcommand*{\Mult}{\mathcal M}
\newcommand*{\Coarse}{\mathcal{C}} 
\newcommand*{\Compactified}{\mathcal{S}}  
\newcommand{\C}{\mathbb{C}}
\newcommand{\N}{\mathbb{N}}
\newcommand{\Z}{\mathbb{Z}}
\newcommand{\R}{\mathbb{R}}
\newcommand{\T}{\mathbb{T}}

\newcommand*{\Hils}[1][H]{\mathcal #1}

\newcommand*{\Cont}{\mathrm C}
\newcommand*{\Contc}{\mathrm{C_c}}
\newcommand*{\Contb}{\mathrm{C_b}}

\newcommand*{\Star}{$^*$\nobreakdash-}
\newcommand*{\nb}{\nobreakdash}
\newcommand*{\Cst}{\mathrm C^*}
\newcommand*{\Wst}{\mathrm W^*}
\newcommand*{\Cred}{\mathrm C^*_\mathrm r}

\newcommand{\Comp}{\mathbb K}
\newcommand{\Bound}{\mathbb B}
\newcommand{\Mat}{\mathbb M}

\makeatletter
\newsavebox\myboxA
\newsavebox\myboxB
\newlength\mylenA

\newcommand*\xoverline[2][0.75]{%
    \sbox{\myboxA}{$\m@th#2$}%
    \setbox\myboxB\null
    \ht\myboxB=\ht\myboxA%
    \dp\myboxB=\dp\myboxA%
    \wd\myboxB=#1\wd\myboxA
    \sbox\myboxB{$\m@th\overline{\copy\myboxB}$}
    \setlength\mylenA{\the\wd\myboxA}
    \addtolength\mylenA{-\the\wd\myboxB}%
    \ifdim\wd\myboxB<\wd\myboxA%
       \rlap{\hskip 0.5\mylenA\usebox\myboxB}{\usebox\myboxA}%
    \else
        \hskip -0.5\mylenA\rlap{\usebox\myboxA}{\hskip 0.5\mylenA\usebox\myboxB}%
    \fi}
\makeatother

\newcommand*{\abs}[1]{\lvert#1\rvert}
\newcommand*{\norm}[1]{\lVert#1\rVert}

\newcommand*{\cl}[1]{\xoverline{#1}}
\newcommand*{\cmp}[1]{\xoverline{#1}}
\newcommand*{\bdy}[1]{\partial #1}
\newcommand{\blank}{{\llcorner\!\lrcorner}}

\newcommand*{\Id}{\mathrm{id}}

\DeclareMathOperator{\supp}{supp}

\DeclareMathOperator{\Aut}{Aut}
\DeclareMathOperator{\Ad}{Ad}

\DeclareMathOperator{\const}{const}

\begin{document}
\title{Noncommutative coarse geometry}
\author{Tathagata Banerjee}
\email{tathagata@mathematik.uni-goettingen.de}
\author{Ralf Meyer}
\email{rmeyer2@uni-goettingen.de}

\address{Mathematisches Institut,
  Georg-August Universit\"at G\"ottingen, Bunsenstra{\ss}e 3--5,
  37073 G\"ottingen, Germany}



\begin{abstract}
  We use compactifications of \(\Cst\)\nb-algebras to introduce
  noncommutative coarse geometry.  We transfer a noncommutative
  coarse structure on a \(\Cst\)\nb-algebra with an action of a
  locally compact Abelian group by translations to Rieffel
  deformations and prove that the resulting noncommutative coarse
  spaces are coarsely equivalent.  We construct a noncommutative
  coarse structure from a cocompact continuously square-integrable
  action of a group and show that this is coarsely equivalent to the
  standard coarse structure on the group in question.  We define
  noncommutative coarse maps through certain completely positive
  maps that induce \Star{}homomorphisms on the boundaries of the
  compactifications.  We lift \Star{}homomorphisms between
  separable, nuclear boundaries to noncommutative coarse maps and
  prove an analogous lifting theorem for maps between the metrisable
  boundaries of ordinary locally compact spaces.
\end{abstract}
\maketitle

\section{Introduction}
\label{sec:intro}

Let~\((X,d)\) be a metric space.  The topology on~\(X\) associated
to~\(d\) encodes its small-scale features.  In contrast, coarse
geometry studies the large-scale geometry of the metric space,
disregarding any phenomena that occur only on some finite
length-scale.  For instance, if a finitely generated discrete
group~\(G\) acts on~\((X,d)\) properly, cocompactly and by
isometries, then~\(G\) with any word-length metric~\(d_G\) is
coarsely equivalent to~\((X,d)\), that is, coarse geometry does not
distinguish the metric spaces \((G,d_G)\) and~\((X,d)\).
Quantum mechanics allows noncommutative spaces
where the coordinate functions no longer commute.  Is large-scale
geometry possible also for such ``spaces''?  Intuitively,
quantum phenomena only occur on small length scales.  So coarse
geometry for noncommutative spaces should be possible, and the
physical ones should even be coarsely equivalent to ordinary
spaces.  We test this intuition on Rieffel deformations, which
provide several toy models used in quantum physics to model quantum
spacetimes.  Intuitively, Rieffel deformation only operates on some
finite length scale depending on the deformation parameter.  So it
should not affect the coarse geometry of a space.

We are going to define noncommutative coarse spaces
and noncommutative coarse maps and show that Rieffel deformation
yields coarsely equivalent noncommutative coarse spaces.  For
instance, the Moyal plane, with a canonical noncommutative coarse
structure, is coarsely equivalent to the classical Euclidean plane.

How should we extend coarse geometry to noncommutative spaces?
The coarse geometric features of a metric~\(d\) on a space~\(X\) may
be encoded through the family of all subsets of~\(X\times X\) on
which the distance function~\(d\) is uniformly bounded; these
subsets are called \emph{controlled}.  Two metrics on~\(X\) are
\emph{coarsely equivalent} if they have the same controlled subsets.
There is, however, no obvious noncommutative analogue of controlled
subsets.
Connes
describes Riemannian metrics in noncommutative geometry
through spectral triples.  A spectral triple comes with a dense
subalgebra -- the domain of the derivation \([D,\blank]\) -- which is
complete in a suitable Lipschitz norm, and
Rieffel~\cite{Rieffel:Metrics_state} encodes a metric on a
noncommutative space through such a densely defined Lipschitz norm on
a \(\Cst\)\nb-algebra.  The existing literature in this direction is
mostly restricted to the unital case, however, and studies the
induced topology on the state space, that is, the
small-scale features of the metric.  Since
any bounded metric space is coarsely equivalent to the one-point
space, noncommutative coarse geometry is only concerned with
non-unital \(\Cst\)\nb-algebras.  We shall, therefore, define coarse
geometric structures on \(\Cst\)\nb-algebras differently.

We work with compactifications, which are another way to capture the
coarse geometry of a metric.  The \emph{Higson compactification} of
a metric space~\((X,d)\) is the spectrum of the commutative
\(\Cst\)\nb-algebra of all continuous functions \(f\colon X\to\C\)
for which the functions
\[
\Delta_R f(x) \defeq \sup \{ \abs{f(y)-f(x)} \mid d(x,y)\le R\}
\]
vanish at infinity for all \(R\ge0\); we briefly say that such
functions \emph{oscillate slowly}.  Two metrics give the same
Higson compactification if and only if they give the same coarse
structure; this is proved by Roe~\cite{Roe:Lectures}.  There are
coarse structures that do not come from a metric
and which cannot be obtained back from their associated Higson
compactification.  But since most coarse structures of interest come
either from a metric or directly from a compactification, it seems
legitimate to replace coarse structures by compactifications to
achieve a noncommutative generalisation.

A \emph{compactification} for a \(\Cst\)\nb-algebra~\(A\) is a
unital \(\Cst\)\nb-algebra~\(\cmp{A}\) with an isomorphism
from~\(A\) onto an essential ideal in~\(\cmp{A}\).  Any
such~\(\cmp{A}\) embeds uniquely into the multiplier
algebra~\(\Mult(A)\).  Hence a noncommutative compactification is
equivalent to a unital subalgebra~\(\cmp{A}\) of~\(\Mult(A)\) that
contains~\(A\).  A \emph{noncommutative coarse space} is a
\(\Cst\)\nb-algebra together with such a compactification.

How does Rieffel deformation affect such a noncommutative coarse
space?  Let~\(\Psi\) be a \(2\)\nb-cocycle on a locally compact
Abelian group~\(\Gamma\) and let~\(\Gamma\) act continuously on a
locally compact space \(X\).  Let~\(\cmp{X}\) be a compactification
of~\(X\) such that the action of~\(\Gamma\) on~\(X\) extends
continuously to~\(\cmp{X}\) and the induced action on the boundary
\(\bdy{X} \defeq \cmp{X}\setminus X\) is trivial; then we say
that~\(\Gamma\) acts on the coarse space~\(X\) \emph{by
  translations}.  Rieffel deformation for a given cocycle~\(\Psi\)
is an exact functor \(A\mapsto A^\Psi\) on the category of
\(\Gamma\)-\(\Cst\)-algebras; this follows from Kasprzak's
description of Rieffel deformation
in~\cite{Kasprzak:Rieffel_deformation}.  It fixes
\(\Cst\)\nb-algebras with a trivial action.  Hence the extension
\[
\Cont_0(X) \into \Cont(\cmp{X}) \prto \Cont(\bdy{X})
\]
induces another extension
\[
\Cont_0(X)^\Psi \into \Cont(\cmp{X})^\Psi \prto
\Cont(\bdy{X})^\Psi
\]
with \(\Cont(\bdy{X})^\Psi = \Cont(\bdy{X})\).  This provides a
coarse structure on the Rieffel deformed \(\Cst\)\nb-algebra
\(\Cont_0(X)^\Psi\) with the same commutative
boundary~\(\Cont(\bdy{X})\).  The same argument works when we
deform a noncommutative coarse space \(A\idealin \cmp{A}\),
provided~\(\Gamma\) acts trivially on the boundary \(\bdy{A} =
\cmp{A}/A\).

The Higson corona \(\bdy{X} = \cmp{X}\setminus X\) is an important
invariant of a coarse space.  If the boundaries are second
countable, then any continuous map between the boundaries lifts to a
coarse map between the interiors
(Theorem~\ref{the:lift_metrisable}).  Thus the boundary is a
complete invariant up to coarse equivalence when it is second
countable.  It is unclear, however, whether this remains true for
larger boundaries such as those coming from a proper metric on the
interior.  We therefore want some kind of map between \(\Cont_0(X)\)
and~\(\Cont_0(X)^\Psi\) that induces the identity
\(\Cont(\bdy{X})^\Psi = \Cont(\bdy{X})\) between the boundary
quotient \(\Cst\)\nb-algebras.  We cannot expect this to be an
ordinary \Star{}homomorphism: in coarse geometry,
we often need discontinuous coarse maps.

Our definition of a noncommutative coarse map is inspired by a
construction by Kaschek, Neumaier and
Waldmann~\cite{Kaschek-Neumaier-Waldmann:Positivity} for
transferring states on a \(\Cst\)\nb-algebra to a Rieffel
deformation.  We define a noncommutative coarse map between two
noncommutative coarse spaces \((A,\cmp{A})\) and \((B,\cmp{B})\) as
a completely positive, contractive, strictly continuous map
\(\varphi\colon A\to B\) such that the resulting strictly continuous
extension \(\cmp{\varphi}\colon \Mult(A)\to\Mult(B)\) is unital,
maps \(\cmp{A}\) to~\(\cmp{B}\), and induces a \Star{}homomorphism
\(\bdy{A}\to\bdy{B}\).  Such a map is behind the transfer of states
in~\cite{Kaschek-Neumaier-Waldmann:Positivity}.  Two noncommutative
coarse maps from \(A\idealin \cmp{A}\) to \(B\idealin \cmp{B}\) are
\emph{close} if their difference maps~\(\cmp{A}\) into~\(B\).  In
Kasprzak's approach to Rieffel deformation
in~\cite{Kasprzak:Rieffel_deformation}, it is easy to build
noncommutative coarse maps \(A\to A^\Psi\) and back whose composites
are close to the identity maps on \(A\) and~\(A^\Psi\), whenever the
induced group action on the boundary~\(\bdy{A}\) is trivial.  These
form a coarse equivalence between \(A\) and~\(A^\Psi\).  In
particular, this applies to the Moyal plane and the classical plane
with their canonical coarse geometries.

The construction above provides rather well-behaved quantisation
maps for Rieffel deformations.  For the Moyal plane, this is a map
from \(\Cont_0(\R^2)\) to \(\Comp(L^2\R)\) that is completely
positive, contractive, and strictly continuous with a unital
extension \(\Contb(\R^2)\to\Bound(L^2\R)\), and such that the
induced map from \(\Contb(\R^2)/\Cont_0(\R^2)\) to the Calkin
algebra is a \Star{}homomorphism on the huge, non-separable
\(\Cst\)\nb-algebra \(\Cont(\bdy{\R^2})\).  All functions that are
continuous on the usual ball compactification of~\(\R^2\) oscillate
slowly and are therefore contained in~\(\Cont(\cmp{\R^2})\).  Thus
shrinking the boundary gives an extension of the Moyal
plane~\(\Comp(L^2\R)\) by~\(\Cont(\T)\), the continuous functions on
the circle.  This extension is the Toeplitz extension by the results
of~\cite{Coburn-Xia:Toeplitz_Rieffel}.  Thus the noncommutative
coarse structure on the Moyal plane~\(\Comp(L^2\R)\) given by the
Toeplitz \(\Cst\)\nb-algebra extension is coarsely equivalent to the
coarse structure on the classical Moyal plane~\(\Cont_0(\R^2)\)
given by Rieffel deformation of the ball compactification.

The relationship between noncommutative coarse maps and ordinary
coarse maps between coarse spaces is not yet clear.  Any coarse map
between commutative coarse spaces is close to a noncommutative
coarse map as defined above, by replacing a discontinuous map by a
continuous map taking values in probability measures.  We failed,
however, to prove the converse, that is, it is conceivable that
there are more noncommutative coarse maps than ordinary coarse maps.

As another example besides Rieffel deformations, we construct a
compactification of a \(\Cst\)\nb-algebra~\(A\) from a cocompact
continuously square-integrable group action on~\(A\).  This
compactification is coarsely equivalent to the group that acts.
Continuously square-integrable actions are slightly more general
than the ``proper'' actions defined by
Rieffel~\cite{Rieffel:Proper}.  So this construction is a
noncommutative analogue of the canonical coarse structure on a
cocompact proper \(G\)\nb-space (see
Example~\ref{exa:coarse_group}).

For noncommutative coarse spaces with separable, nuclear boundary,
we prove that the boundary is a complete invariant up to coarse
equivalence.  In this case, any unital \Star{}homomorphism between
the boundaries lifts to a noncommutative coarse map.  This is proved
using quasi-central approximate units.

\section{Compactifications as coarse structures}
\label{sec:nc_compactification}

A coarse structure on a locally compact space~\(X\) is a family of
subsets of~\(X\times X\), called \emph{controlled}, subject to some
axioms (see~\cite{Roe:Lectures}).  For instance, a subset
\(E\subseteq X\times X\) is controlled with respect to a
metric~\(d\) on~\(X\) if and only if~\(d|_E\) is bounded.

\begin{example}
  \label{exa:coarse_group}
  Let~\(G\) be a locally compact group and let~\(X\) be a locally
  compact space with a continuous, proper, cocompact action
  of~\(G\).  For instance, we may take \(X=G\) with the action by
  left translation.  There is a unique proper coarse structure
  on~\(X\) for which any controlled subset is contained in a
  \(G\)\nb-invariant controlled subset.  Namely, a subset~\(E\) of
  \(X\times X\) is controlled if and only if it is contained in
  \(E_K \defeq \{(g x, g y) \mid g\in G,\ x,y\in K\}\) for some
  compact subset \(K\subseteq X\).  This is a proper coarse
  structure because~\(X\) is locally compact, and any controlled
  subset is contained in a \(G\)\nb-invariant one.  Conversely, in a
  proper coarse structure on~\(X\) with enough \(G\)\nb-invariant
  controlled sets, the sets \(E_K\subseteq X\times X\) for compact
  \(K\subseteq G\) must be controlled.  And if there were more
  controlled subsets, some non-compact subsets would have to be
  bounded, which is forbidden for proper coarse structures.
\end{example}

If~\(G\) is a finitely generated discrete group, then the coarse
structure in Example~\ref{exa:coarse_group} is the metric coarse
structure for any word-length metric on~\(G\).  If \(G=\R^{2n}\),
then the coarse structure in Example~\ref{exa:coarse_group} is the
metric coarse structure for the Euclidean metric on~\(\R^{2n}\).

The definition of a coarse structure does not carry over to
noncommutative spaces.  Therefore, we shall study compactifications
instead of coarse structures:

\begin{definition}
  \label{def:compactification}
  Let~\(X\) be a locally compact space.  A \emph{compactification}
  of~\(X\) is a compact space~\(\cmp{X}\) with a homeomorphism
  from~\(X\) onto a dense, open subset of~\(\cmp{X}\).
\end{definition}

We will explain below why compactifications are a reasonable
substitute for coarse spaces.  First we discuss the noncommutative
version of compactifications.

\begin{definition}
  \label{def:nc_compactification}
  Let~\(A\) be a \(\Cst\)\nb-algebra.  A (noncommutative)
  \emph{compactification} of~\(A\) is a unital
  \(\Cst\)\nb-algebra~\(\cmp{A}\) with a \Star{}isomorphism
  from~\(A\) onto an essential ideal in~\(\cmp{A}\).  Being
  \emph{essential} means that for any \(a\in \cmp{A}\) with \(a\neq
  0\) there is \(x\in A\) with \(a\cdot x\neq 0\).  The quotient
  \(\bdy{A} \defeq \cmp{A}/A\) is called the boundary or
  \emph{corona algebra} of the compactification.
\end{definition}

The \emph{multiplier algebra} \(\Mult(A)\) of~\(A\) is unital and
contains~\(A\) as an essential ideal, so it is a compactification.
It is the largest compactification of~\(A\):

\begin{lemma}
  \label{lem:compafication_in_MA}
  Let \(A\idealin \cmp{A}\) be a compactification.  The identity map
  on~\(A\) extends uniquely to an isomorphism from~\(\cmp{A}\) onto
  a unital \(\Cst\)\nb-subalgebra of~\(\Mult(A)\) containing~\(A\).
  Conversely, any unital \(\Cst\)\nb-subalgebra of~\(\Mult(A)\)
  containing~\(A\) gives a compactification of~\(A\).  Thus we may
  also define a noncommutative compactification of~\(A\) as a unital
  \(\Cst\)\nb-subalgebra of~\(\Mult(A)\) that contains~\(A\).
\end{lemma}

\begin{proof}
  A \Star{}homomorphism \(\iota\colon \cmp{A}\to\Mult(A)\) that
  extends the canonical inclusion \(A\idealin \Mult(A)\) must map
  \(x\in\cmp{A}\) to the multiplier defined by \(a\mapsto x\cdot
  a\in A\) for \(a\in A\).  This indeed defines a unital
  \Star{}homomorphism~\(\iota\), which is injective because the
  ideal~\(A\) in~\(\cmp{A}\) is essential.  So it
  identifies~\(\cmp{A}\) with a unital \(\Cst\)\nb-subalgebra
  of~\(\Mult(A)\) containing~\(A\).  Conversely, \(A\) is an
  essential ideal in~\(\Mult(A)\) and hence in any unital
  \(\Cst\)\nb-subalgebra of~\(\Mult(A)\) that contains~\(A\).
\end{proof}

\begin{lemma}
  \label{lem:compactification}
  Let~\(X\) be a locally compact space.  There is a bijection
  between isomorphism classes of compactifications of \(X\)
  and\/~\(\Cont_0(X)\).
\end{lemma}

\begin{proof}
  Let \(X\subset \cmp{X}\) be a compactification of~\(X\).  Then
  \(\Cont(\cmp{X})\) is a unital \(\Cst\)\nb-algebra and extension
  by zero is a \Star{}isomorphism from~\(\Cont_0(X)\) onto an ideal
  in~\(\Cont(\cmp{X})\).  This ideal is essential because~\(X\) is
  dense in~\(\cmp{X}\).
  Conversely, let \(\Cont_0(X)\idealin A\) be a compactification.
  By Lemma~\ref{lem:compafication_in_MA}, we may identify~\(A\) with
  a \(\Cst\)\nb-subalgebra of \(\Mult(\Cont_0(X)) \cong \Contb(X)\),
  so~\(A\) is commutative.  Thus \(A\cong\Cont(\cmp{X})\) for some
  compact space~\(\cmp{X}\).  The ideal \(\Cont_0(X)\idealin A\)
  corresponds to an open subset in~\(\cmp{X}\), together with a
  homeomorphism between this open subset and~\(X\).  This open
  subset in~\(\cmp{X}\) is dense because \(\Cont_0(X)\idealin A\) is
  essential.
\end{proof}

A metric gives rise to a compactification as follows:

\begin{example}
  \label{exa:Higson_compactification}
  Let~\(X\) be a locally compact space and let~\(d\) be a
  continuous, proper metric on~\(X\).  A bounded continuous function
  \(f\colon X\to \C\) \emph{oscillates slowly} if for all \(R>0\)
  and \(\varepsilon>0\) there is a compact subset \(K\subseteq X\)
  such that \(\abs{f(x)-f(y)}<\varepsilon\) if \(d(x,y)\le R\) and
  \(x,y\notin K\); the name ``slowly oscillating'' comes
  from~\cite{Willett:Index_band-dominated}.  The slowly
  oscillating functions form a unital \(\Cst\)\nb-subalgebra of
  \(\Contb(X)=\Mult(\Cont_0(X))\), which contains~\(\Cont_0(X)\).
  Hence they provide a compactification of~\(\Cont_0(X)\).  The
  corresponding compactification of~\((X,d)\) is called \emph{Higson
    compactification}.

  More generally, slowly oscillating can be defined on any coarse
  space.  So any coarse space has a (Higson) compactification.
\end{example}

Let \(X\subseteq \cmp{X}\) be a compactification
of~\(X\) and \(\bdy{X} \defeq \cmp{X}\setminus X\).  Call a subset
\(E\subseteq X\times X\) \emph{controlled} if the closure of~\(E\)
in \(\cmp{X}\times\cmp{X}\) meets the set
\[
\cmp{X}\times\cmp{X} \setminus X\times X
= \cmp{X}\times\bdy{X} \cup \bdy{X}\times\cmp{X}
\]
only in the diagonal; that is, if \((x_i,y_i)_{i\in I}\) is a net
in~\(E\) that converges in \(\cmp{X}\times\cmp{X}\) with \(\lim
x_i\in\bdy{X}\) or \(\lim y_i\in\bdy{X}\), then \(\lim x_i = \lim
y_i\).

\begin{proposition}
  \label{pro:compactification2coarse_metrisable}
  The controlled subsets associated to a compactification
  \(X\subseteq \cmp{X}\) form a coarse structure on~\(X\) whose
  bounded subsets are exactly the relatively compact ones.  The
  coarse structure
  on~\(X\) induced by~\(\cmp{X}\) is proper if \(X\subseteq
  \cmp{X}\) is the Higson compactification of a proper coarse
  structure or if~\(X\) is \(\sigma\)\nb-compact and~\(\bdy{X}\) is
  second countable.  In the second case, the resulting Higson
  compactification is isomorphic to the original
  compactification~\(\cmp{X}\).  In general, the coarse structure
  defined by a compactification is proper if and only if the
  separated, locally closed subsets \(\{(x,x)\mid x\in X\}\) and
  \begin{equation}
    \label{eq:bad_points_for_control}
    \{(x,y)\in \cmp{X}\times \cmp{X} \mid x\neq y
    \text{ and }x\in \bdy{X} \text{ or }y \in \bdy{X}\}
  \end{equation}
  of \(\cmp{X}\times \cmp{X}\) are separated by neighbourhoods.
\end{proposition}

\begin{proof}
  All this would follow from \cite{Roe:Lectures}*{Theorem~2.27 and
    Proposition~2.48}, but the assumptions made there do not suffice
  for the proof.  The correction in~\cite{Roe:Correction_Lectures}
  explains that the theorem works if~\(\cmp{X}\) is second
  countable.  The main issue is whether there is a controlled
  neighbourhood of the diagonal.  Lemma~\ref{lem:controlled nbd}
  shows this if~\(X\) is \(\sigma\)\nb-compact and~\(\bdy{X}\) is
  second countable; then~\(X\) is paracompact as assumed
  throughout~\cite{Roe:Lectures}.  All functions
  in~\(\Cont(\cmp{X})\) are slowly oscillating for the coarse
  structure defined by~\(\cmp{X}\).  The Higson compactification is
  canonically isomorphic to the original
  compactification~\(\cmp{X}\) if and only if all slowly oscillating
  functions belong to~\(\Cont(\cmp{X})\).  The proof
  in~\cite{Roe:Lectures} that this is so if~\(\cmp{X}\) is second
  countable goes through under our assumptions because any point
  in~\(\bdy{X}\) is a limit of a sequence in~\(X\) by
  Lemma~\ref{lem:convergence}.\ref{en:convergence3}.

  If~\(\cmp{X}\) is the Higson
  compactification of a proper coarse structure, then all controlled
  subsets of the original coarse structure remain controlled in the
  coarse structure induced by~\(\cmp{X}\).  Hence a controlled
  neighbourhood for the original coarse structure remains controlled
  for the new coarse structure, so that the latter is again proper.

  The subset of \(\cmp{X}\times\cmp{X}\)
  in~\eqref{eq:bad_points_for_control} and \(\{(x,x)\mid x\in X\}\)
  are separated, that is, each is disjoint from the other's closure.
  Being separated by neighbourhoods means that there are disjoint
  open subsets \(V\) and~\(U\) of \(\cmp{X}\times\cmp{X}\) that
  contain them.  Then \(U\cap X\times X\) is a controlled
  neighbourhood of the diagonal in~\(X\times X\).  Conversely, a
  controlled neighbourhood of the diagonal in~\(X\times X\) contains
  an open controlled neighbourhood~\(U\).  Thus \(\cl{U}\cap
  \cmp{X}\times\cmp{X} \setminus X\times X\) is contained in the
  diagonal, so that the open subset \(V\defeq
  (\cmp{X}\times\cmp{X})\setminus \cl{U}\) contains the subset
  in~\eqref{eq:bad_points_for_control}.  So the open subsets \(U\)
  and~\(V\) separate our two subsets.
\end{proof}

\begin{proposition}[\cite{Roe:Lectures}*{Proposition~2.47}]
  \label{pro:metric_coarse_compactification_equivalent}
  Let~\(d\) be a continuous proper metric on~\(X\).  Build its
  Higson compactification as in
  Example~\textup{\ref{exa:Higson_compactification}}.  A subset
  \(E\subseteq X\times X\) is controlled with respect to the Higson
  compactification if and only if~\(d|_E\) is bounded.
\end{proposition}

Proposition~\ref{pro:metric_coarse_compactification_equivalent}
shows that the coarse structure and the Higson compactification
associated to a proper metric on a space~\(X\) contain the same
information, that is, one determines the other uniquely.  There are
coarse structures that do not come from any compactification.  But
the most important examples are those defined by metrics or by
compactifications with metrisable boundary.  In both cases, the
coarse structure and the compactification determine each other by
Propositions \ref{pro:compactification2coarse_metrisable} and
\ref{pro:metric_coarse_compactification_equivalent}.  Hence the
following definition seems legitimate:

\begin{definition}
  \label{def:nc_space}
  A \emph{noncommutative coarse space} is a
  \(\Cst\)\nb-algebra~\(A\) with a compactification
  \(A\idealin \cmp{A}\).
\end{definition}

The following example shows that any unital \(\Cst\)\nb-algebra
appears as a corona algebra in some noncommutative coarse space:

\begin{example}
  \label{exa:standard_coarse_space}
  Let~\(B\) be a unital \(\Cst\)\nb-algebra.  The \emph{cone
    over~\(B\)} is the noncommutative coarse space \(\Cont_0(\N,B)
  \idealin \Cont(\N^+,B)\).  Its corona algebra
  \(\Cont(\N^+,B)\mathbin{/}\Cont_0(\N,B)\) is naturally isomorphic
  to~\(B\) by evaluation at~\(\infty\).
\end{example}

\begin{definition}
  \label{def:translation}
  Let \(A\idealin \cmp{A}\) be a noncommutative coarse space.  An
  automorphism \(\alpha\in\Aut(A)\) is a \emph{translation} if its
  canonical extension to~\(\Mult(A)\) maps \(\cmp{A}\subseteq
  \Mult(A)\) to itself and induces the identity map on \(\bdy{A}
  \defeq \cmp{A}/A\).
\end{definition}

\begin{example}
  \label{exa:nc_compactify_group}
  Let~\(G\) be a locally compact group and let~\(A\) be a
  \(\Cst\)\nb-algebra with a continuous action~\(\alpha\) of~\(G\).  Let
  \[
  \cmp{A}_G \defeq
  \{x\in \Mult(A) \mid \alpha_g(x) -x \in A\text{ for all }g\in G\}.
  \]
  This is a compactification of~\(A\), and it is the
  largest one such that~\(G\) acts by translations, that is, it acts
  trivially on \(\bdy{A}_G \defeq \cmp{A}_G/A\).  Since~\(G\) acts
  continuously on \(A\) and~\(\bdy{A}_G\), it also acts continuously
  on~\(\cmp{A}_G\) by~\cite{Brown:Continuity}.
\end{example}

\begin{lemma}
  \label{lem:Higson_group}
  Let \(A=\Cont_0(G)\) with the \(G\)\nb-action by right
  translation.  Then \(\cmp{A}_G\subseteq \Contb(G)\) is the
  \(\Cst\)\nb-algebra of slowly oscillating functions for the coarse
  structure on~\(G\) in Example~\textup{\ref{exa:coarse_group}}.
\end{lemma}

\begin{proof}
  Let \(f\in\Contb(G)\).  By definition, \(f\) oscillates slowly for
  the coarse structure of Example~\ref{exa:coarse_group} if and only
  if \(\abs{f(gx) - f(gy)}\to0\) for \(g\to\infty\), uniformly
  for~\(x,y\) in any compact subset of~\(G\).  Replacing \(g\)
  by~\(gy\), we see that we may as well take \(y=1\).  Let
  \(\alpha_x f(g) \defeq f(g x)\).  We may rewrite the slow
  oscillation condition for fixed~\(x\) as \(\alpha_x f - f\in
  \Cont_0(G)\).  Thus all slowly oscillating functions belong to
  \(\cmp{\Cont_0(G)}_G\).  For the converse, we use the automatic
  continuity of the \(G\)\nb-action on~\(\cmp{A}_G\), which follows
  from~\cite{Brown:Continuity}.  If \(f\in\cmp{\Cont_0(G)}_G\), then
  the map \(x\mapsto \alpha_x f -f\) is continuous and hence maps a
  compact subset of~\(G\) to a compact subset~\(K\) of
  \(\Cont_0(G)\subseteq \Cont(G^+)\).  By the Arzel\`a--Ascoli
  Theorem, the functions \(\alpha_x f -f\) for \(x\in K\) are
  uniformly continuous on the one-point compactification~\(G^+\).
  Thus they vanish uniformly at infinity.  This means that~\(f\)
  oscillates slowly.
\end{proof}

\section{Rieffel deformation of coarse structures}
\label{sec:Rieffel_deform}

Our main examples of noncommutative coarse spaces are Rieffel
deformations of commutative coarse spaces.  We shall use Kasprzak's
description of Rieffel deformations
in~\cite{Kasprzak:Rieffel_deformation} because it simplifies the
proof of functorial properties and is slightly more general.
Mostly, we may treat Rieffel deformation as a black box and use only
some functorial properties listed below.  This may help to extend
the theory to other situations.  A candidate are the deformations
in~\cite{Bieliavsky-Gayral:Deformation_Kaehler}; but we have not
checked whether the following works in that situation.

Let~\(G\) be an Abelian, locally compact group and let~\(\Psi\) be a
continuous \(2\)\nb-cocycle on~\(G\).  A \(G\)\nb-\(\Cst\)-algebra
is a \(\Cst\)\nb-algebra with a (strongly) continuous action
of~\(G\) by automorphisms.  Rieffel deformation with respect
to~\(\Psi\) maps a \(G\)\nb-\(\Cst\)-algebra~\(A\) to another
\(G\)\nb-\(\Cst\)-algebra~\(A^\Psi\).  This construction has the
following properties:
\begin{enumerate}[label=\textup{(RD\arabic*)}]
\item \label{enum:RD1} Rieffel deformation is functorial for
  \(G\)\nb-equivariant \Star{}homomorphisms and morphisms
  (nondegenerate \Star{}homomorphisms to multiplier algebras); this
  functor preserves injectivity of morphisms, see
  \cite{Kasprzak:Rieffel_deformation}*{Proposition~3.8}.
\item \label{enum:RD2} Rieffel deformation is exact, that is, it
  maps an extension of \(G\)\nb-\(\Cst\)\nb-algebras again to an
  extension of \(G\)\nb-\(\Cst\)\nb-algebras, see
  \cite{Kasprzak:Rieffel_deformation}*{Theorem~3.9}.
\item \label{enum:RD3} If~\(G\) acts trivially on~\(A\), then
  \(A^\Psi=A\); this is not stated
  in~\cite{Kasprzak:Rieffel_deformation}, but easy to prove, see
  Lemma~\ref{lem:Rieffel_trivial} below.
\item \label{enum:RD4} If \(\Psi_1,\Psi_2\) are two cocycles then
  \((A^{\Psi_1})^{\Psi_2} = A^{\Psi_1\Psi_2}\) with the product
  cocycle \(\Psi_1 \Psi_2\), see
  \cite{Kasprzak:Rieffel_deformation}*{Lemma~3.5}.
\item \label{enum:RD5} If~\(\Psi\) is the unit cocycle
  (constant~\(1\)), then \(A^\Psi = A\); this is trivial.
\end{enumerate}

\begin{theorem}
  \label{the:Rieffel_deform_action}
  Let \(A\idealin \cmp{A}\) be a noncommutative coarse space,
  equipped with a continuous action of~\(G\); that is, \(G\) acts
  continuously on~\(\cmp{A}\), leaving the ideal~\(A\) invariant.
  Then \(A^\Psi\idealin \cmp{A}{}^\Psi\) is a noncommutative coarse
  space, and its corona algebra is the Rieffel
  deformation~\(\bdy{A}^\Psi\).
\end{theorem}

\begin{proof}
  The ideal inclusion \(A\idealin
  \cmp{A}\) is equivalent to a morphism \(\cmp{A}\to A\), that is,
  to a nondegenerate \Star{}homomorphism \(\cmp{A}\to\Mult(A)\);
  this morphism is injective if and only if~\(A\) is an essential
  ideal in~\(\cmp{A}\), compare Lemma~\ref{lem:compafication_in_MA}.
  Thus a noncommutative coarse space is the same as an extension
  of \(\Cst\)-algebras \(A\into \cmp{A} \prto \bdy{A}\) such
  that~\(\cmp{A}\) is unital and the resulting morphism \(\cmp{A}\to
  A\) is injective.

  By assumption, the extension \(A\into \cmp{A} \prto \bdy{A}\) is
  an extension of \(G\)-\(\Cst\)-algebras and the injective morphism
  \(\cmp{A}\to A\) is \(G\)\nb-equivariant.  By \ref{enum:RD1}
  and~\ref{enum:RD2}, Rieffel deformation maps this to an extension
  \(A^\Psi\into \cmp{A}{}^\Psi \prto \bdy{A}^\Psi\), such that the
  resulting morphism \(\cmp{A}{}^\Psi\to A^\Psi\) is injective.  That
  is, \(A^\Psi\) is an essential ideal in~\(\cmp{A}{}^\Psi\), and the
  corona algebra \(\cmp{A}{}^\Psi/A^\Psi\) is the Rieffel
  deformation~\(\bdy{A}^\Psi\) of~\(\bdy{A}\).  An
  algebra~\(\cmp{A}\) is unital if and only if there is a
  nondegenerate \Star{}homomorphism \(\C\to \cmp{A}\), which is
  equivariant for the trivial action on~\(\C\).  Rieffel deformation
  maps this to a nondegenerate \Star{}homomorphism \(\C^\Psi\to
  \cmp{A}{}^\Psi\).  Since \(\C^\Psi=\C\) by~\ref{enum:RD3},
  \(\cmp{A}{}^\Psi\) is unital.
\end{proof}

\begin{theorem}
  \label{the:Rieffel_deform_translation_action}
  Let \(A\idealin \cmp{A}\) be a noncommutative coarse space and
  let~\(G\) act on~\(A\) continuously and by translations.
  Then~\(G\) acts continuously on~\(\cmp{A}\).  And \(A^\Psi\idealin
  \cmp{A}{}^\Psi\) is a noncommutative coarse space with the corona
  algebra \(\bdy{A}^\Psi = \bdy{A}\), and \(G\) acts on~\(A^\Psi\)
  by translations.
\end{theorem}

\begin{proof}
  By assumption, the action on~\(A\) extends to an action
  on~\(\cmp{A}\) that induces the trivial action on~\(\bdy{A}\).
  Since the trivial action on~\(\bdy{A}\) and the action on~\(A\)
  are continuous, so is the induced action on~\(\cmp{A}\) by the
  main result of~\cite{Brown:Continuity}.  Hence Rieffel deformation
  makes sense in this case.  Theorem~\ref{the:Rieffel_deform_action}
  shows that \(A^\Psi\idealin \cmp{A}{}^\Psi\) is a noncommutative
  coarse space with the corona algebra~\(\bdy{A}^\Psi\).
  Since~\(G\) acts trivially on~\(\bdy{A}\), \ref{enum:RD3} gives
  \(\bdy{A}^\Psi = \bdy{A}\) with the trivial action of~\(G\).
  So~\(G\) acts on~\(A^\Psi\) by translations.
\end{proof}

\begin{lemma}
  \label{lem:Rieffel_trivial}
  Let~\(G\) act trivially on~\(A\).  Then \(A^\Psi=A\) as
  subalgebras of \(\Mult(A\rtimes G)\); even more, the deformed dual
  action on \(A\rtimes G\) used by Kasprzak to define the Rieffel
  deformation is the same as the original dual action.
\end{lemma}

\begin{proof}
  The deformation of the dual action conjugates the automorphism by
  some unitary elements of \(\Cont_0(\hat{G}) \cong \Cst(G)
  \subseteq \Mult(A\rtimes G)\).  Since the action of~\(G\) on~\(A\)
  is trivial, \(A\rtimes G\cong A\otimes \Cst(G)\), so \(\Cst(G)\)
  is contained in the centre of the multiplier algebra.  Hence
  conjugating by unitaries in~\(\Cst(G)\) is the identity
  automorphism, and so the deformation does, in fact, not change the
  dual action.
\end{proof}

\subsection{The coarse Moyal plane}
\label{sec:coarse_Moyal}

Let \(G\defeq \R^{2n}\) and \(A\defeq \Cont_0(\R^{2n})\) with the
translation action, \(\alpha_g f(x) \defeq f(x-g)\) for all
\(x,g\in\R^{2n}\).  Equip~\(G\) with the usual Euclidean metric and
let~\(\cmp{A}\) be the \(\Cst\)\nb-algebra of slowly oscillating
functions, see Example~\ref{exa:Higson_compactification}.  The
automorphisms~\(\alpha_g\) of~\(A\) are indeed ``translations'' as
in Definition~\ref{def:translation}.  Hence
Theorem~\ref{the:Rieffel_deform_translation_action} allows us to
transport this coarse structure to any Rieffel deformation of~\(A\),
in such a way that the corona algebra stays the same.

If the cocycle~\(\Psi\) comes from a nondegenerate antisymmetric
bilinear map on~\(\R^{2n}\), then the Rieffel deformation for the
translation action on~\(\Cont_0(\R^{2n})\) is a Moyal plane.  Its
underlying \(\Cst\)\nb-algebra is isomorphic to the
\(\Cst\)\nb-algebra of compact operators on the Hilbert space
\(L^2(\R^n)\).  This is well-known for Rieffel's original definition
of his deformation, which is is equivalent to Rieffel's
(see~\cite{Neshveyev:Smooth}).  Thus we
have found some noncommutative compactification of
\(\Comp(L^2\R^n)\), that is, a unital \(\Cst\)\nb-subalgebra
\(\cmp{A}{}^\Psi\) of \(\Bound(L^2\R^n) \cong
\Mult(\Comp(L^2\R^n))\), such that
\(\cmp{A}{}^\Psi/\Comp(L^2\R^n)\), its image in the Calkin algebra,
is commutative.  Namely, \(\cmp{A}{}^\Psi/\Comp(L^2\R^n) \cong
\Cont(\bdy{\R^{2n}})\) is the \(\Cst\)\nb-algebra of continuous
functions on the Higson corona of~\(\R^{2n}\) for the Euclidean
metric.

Next we describe this compactification of the Moyal plane
directly, without any reference to the classical Euclidean plane.
The group~\(\R^{2n}\) acts on~\(L^2\R^n\) by a projective
representation, where one copy of~\(\R^n\) acts by translation, the
other by pointwise multiplication with characters.  This projective
representation on~\(L^2\R^n\)
induces a strongly continuous \(\R^{2n}\)\nb-action~\(\beta\) on
\(B\defeq\Comp(L^2\R^n)\).  This is equivalent to the
\(\R^{2n}\)\nb-action on \(\Cont_0(\R^{2n})^\Psi \cong
\Comp(L^2\R^n)\) as a Rieffel deformation.

\begin{theorem}
  \label{the:coarse_Moyal}
  The coarse structure on the Moyal plane \(B\defeq \Comp(L^2\R^n)\)
  constructed by Rieffel deformation of the classical Euclidean
  plane is
  \[
  \cmp{B} \defeq \{x\in \Bound(L^2\R^n) \mid
  \beta_g(x)-x\in \Comp(L^2\R^n)\ \text{for all}\ g\in \R^{2n}\},
  \]
  the largest compactification for which the
  \(\R^{2n}\)\nb-action~\(\beta\) is by translations; see also
  Example~\textup{\ref{exa:nc_compactify_group}}.
\end{theorem}

\begin{proof}
  Let \(B\idealin \cmp{B}\) be the coarse structure constructed in
  Example~\ref{exa:nc_compactify_group}.  It is the
  largest coarse structure such that~\(\R^{2n}\) acts by
  translations.  Let~\(\Psi\) be the cocycle on~\(\R^{2n}\) such
  that \(B=A^\Psi\) with \(A=\Cont_0(\R^{2n})\).  Thus \(B^{\Psi^*} = A =
  \Cont_0(\R^{2n})\) by \ref{enum:RD4} and~\ref{enum:RD5}.  By
  Theorem~\ref{the:Rieffel_deform_translation_action}, the action
  of~\(\R^{2n}\) on~\(A\) is by translations with respect to the
  compactification~\(\cmp{B}{}^{\Psi^*}\).  Even more, we claim that
  \(\cmp{B}{}^{\Psi^*}\) is the largest compactification of~\(A\)
  for which~\(\R^{2n}\) acts by translations.  Indeed, if
  \(\cmp{A}\) is such a compactification of~\(A\), then~\(\R^{2n}\)
  still acts by translations on~\(\cmp{A}{}^\Psi\) by
  Theorem~\ref{the:Rieffel_deform_translation_action}, so that
  \(\cmp{A}{}^\Psi\subseteq \cmp{B}\); and then \(\cmp{A} =
  \cmp{A}{}^{\Psi\Psi^*} = (\cmp{A}{}^\Psi)^{\Psi^*}\subseteq
  \cmp{B}{}^{\Psi^*}\) by \ref{enum:RD5}, \ref{enum:RD4}
  and~\ref{enum:RD1}.  The largest compactification of~\(\R^{2n}\)
  where~\(\R^{2n}\) acts by translations is the one described in
  Example~\ref{exa:nc_compactify_group}.
  Lemma~\ref{lem:Higson_group} identifies it with the standard
  coarse structure on~\(\R^{2n}\), which is the same as the coarse
  structure from the Euclidean metric on~\(\R^{2n}\).  So \(\cmp{B}
  = (\cmp{B}{}^{\Psi^*})^\Psi \cong \cmp{A}{}^\Psi\) is the Rieffel
  deformation of the Higson compactification of~\(\Cont_0(\R^{2n})\)
  for the standard coarse structure on~\(\R^{2n}\).
\end{proof}

Now let~\(\cmp{\R^{2n}}\) be the usual ball compactification with
the boundary \(\bdy{\R^{2n}} \cong \mathbb{S}^{2n-1}\).  All
continuous functions on~\(\cmp{\R^{2n}}\) oscillate slowly with
respect to the Euclidean metric on~\(\R^{2n}\).  When we apply
Rieffel deformation as above, then we get a compactification of the
Moyal plane~\(\Comp(L^2\R)\) by~\(\Cont(\mathbb{S}^{2n-1})\).  The
Rieffel deformation of~\(\Cont(\cmp{\R^{2n}})\) is already studied
by Coburn and Xia~\cite{Coburn-Xia:Toeplitz_Rieffel}.  They identify
it with the Toeplitz algebra of the unit ball in~\(\C^n\).  In
particular, for \(n=1\) we get the usual Toeplitz algebra, the
universal \(\Cst\)\nb-algebra generated by a single isometry.
Since~\(\Cont(\cmp{\R^{2n}})\) is contained in the Higson
compactification for the Euclidean metric, their Rieffel
deformations are also contained in one another.  So the
compactification described in Theorem~\ref{the:coarse_Moyal} may be
related to Toeplitz operators with discontinuous symbols.

\section{Cocompact continuously square-integrable group actions}
\label{sec:cocompact_si}

Example~\ref{exa:coarse_group} describes a unique
``\(G\)\nb-invariant'' coarse structure on a locally compact
space~\(X\) with a continuous, proper, cocompact action of a locally
compact group~\(G\).  The space~\(X\) equipped with this
coarse structure is coarsely equivalent to~\(G\) with the standard
coarse structure: the orbit inclusions \(G\to X\), \(g\mapsto g\cdot
x\), for \(x\in X\) are coarse equivalences.  We are going to extend
this construction to the noncommutative case.  So let~\(A\) be a
\(\Cst\)\nb-algebra with a continuous action~\(\alpha\) of~\(G\).
We first have to define when~\(\alpha\) is ``proper'' and
``cocompact.''  We interpret ``proper'' as
\emph{continuously square-integrable} as
in~\cite{Meyer:Generalized_Fixed}, see also \cites{Rieffel:Proper,
  Rieffel:Integrable_proper}.

A continuously square-integrable action comes with some additional
data.  We shall use the equivalent characterisation of continuous
square-integrability in \cite{Meyer:Generalized_Fixed}*{Theorem 6.1}
(in the case of the Hilbert \(A\)\nb-module
\(\mathcal{E}=A\)):
a continuously square-integrable structure for~\((A,\alpha,G)\)
is equivalent to a Hilbert module~\(\mathcal{F}\)
over \(A\rtimes_{\mathrm{r},\alpha} G\)
with a \(G\)\nb-equivariant unitary of Hilbert \(A\)\nb-modules
\begin{equation}
  \label{eq:csi_unitary}
  U\colon A\congto
  \mathcal{F}\otimes_{A\rtimes_{\mathrm{r},\alpha}G} L^2(G,A);
\end{equation}
here we use the faithful representation
of the reduced crossed product \(A\rtimes_{\mathrm{r},\alpha} G\) on
the Hilbert \(A\)\nb-module~\(L^2(G,A)\) by \(G\)\nb-equivariant
operators described in~\cite{Meyer:Generalized_Fixed}*{§3}.
Let \(\Mult(A)^G\)
be the \(\Cst\)\nb-algebra
of \(G\)\nb-invariant
multipliers of~\(A\).
The \emph{generalised fixed point algebra} of the continuously
square-integrable action with Hilbert module~\(\mathcal{F}\)
is defined to be \(\Comp(\mathcal{F})\),
viewed as a \(\Cst\)\nb-subalgebra
of~\(\Mult(A)^G\) through the representation
\[
\Comp(\mathcal{F}) \xrightarrow{T\mapsto T\otimes 1}
\Bound\bigl(\mathcal{F} \otimes_{A\rtimes_{\mathrm{r},\alpha} G}
L^2(G,A)\bigr)^G
\xrightarrow[\cong]{\Ad U^*} \Bound(A)^G = \Mult(A)^G.
\]
This contains the construction of a generalised fixed point algebra by
Rieffel~\cite{Rieffel:Proper} as a special case.

Let~\(X\) be a locally compact space with a continuous action
of~\(G\).  Let \(A=\Cont_0(X)\) and let~\(\alpha\) be the induced
action of~\(G\) on~\(A\), which is continuous.  As shown in
\cite{Meyer:Generalized_Fixed}*{§9}, this action is continuously
square-integrable if and only if the action on~\(G\) is proper; the
Hilbert module~\(\mathcal{F}\) is unique in this case; and the
generalised fixed point algebra is \(\Cont_0(X/G)\).  This is unital
if and only if~\(X/G\) is compact.  This justifies the following
definition:

\begin{definition}
  \label{def:cocompact}
  A continuously square-integrable action is \emph{cocompact} if its
  generalised fixed point algebra is unital.
\end{definition}

From now on, we assume the action~\(\alpha\)
to be cocompact continuously square-integrable.  Implicitly, this
fixes a Hilbert module~\(\mathcal{F}\)
over \(B\defeq A\rtimes_{\mathrm{r},\alpha} G\)
and a \(G\)\nb-equivariant
unitary Hilbert \(A\)\nb-module
isomorphism~\(U\)
as in~\eqref{eq:csi_unitary}.  In this situation, we want to define a
coarse structure on~\(A\).

\begin{proposition}
  \label{pro:cocompact_Hilbert_module}
  Let~\(B\) be a \(\Cst\)\nb-algebra and~\(\mathcal{F}\) a Hilbert
  \(B\)\nb-module.  The \(\Cst\)\nb-algebra \(\Comp(\mathcal{F})\)
  is unital if and only if there are \(n\in\N\), a projection \(p\in
  \Mat_n(B)\), and a unitary \(\mathcal{F} \cong p\cdot B^n\).  The
  projection~\(p\) is unique up to Murray--von Neumann equivalence.
\end{proposition}

\begin{proof}
  The second claim is trivial.  The first one is contained in well
  known results if~\(B\) is unital.  See, for instance,
  \cite{Wegge-Olsen:K-theory}*{Theorem 15.4.2} and the remarks
  following it.  If~\(B\) is not unital, then it still follows by
  similar arguments.  If
  \(\mathcal{F}\cong p\cdot B^n\) for a projection
  \(p\in\Mat_n(B)\), then \(\Id_{\mathcal{F}}\) is compact because
  so is~\(p\) and compact operators form an ideal.  Conversely,
  assume that \(\Id_{\mathcal{F}}\) is compact.  Then
  \(\Comp(\mathcal{F})\) is \(\sigma\)\nb-unital and
  hence~\(\mathcal{F}\) is countably generated.  So \(\mathcal{F}
  \cong p_1\cdot \ell^2(\N,B)\) for a projection
  \(p_1\in\Bound\bigl(\ell^2(\N,B)\bigr)\) by the Kasparov
  Stabilisation Theorem.  Actually,
  \(p_1\in\Comp\bigl(\ell^2(\N,B)\bigr)\) because
  \(\Id_{\mathcal{F}}\) is compact.  Since finite matrices are dense
  in the compact operators, there are \(n\in\N\) and \(p_2
  \in\Mat_n(B)\) so that \(\norm{p_2-p_1}<\nicefrac12\).
  Then~\(p_2\) is close to a projection \(p\in\Mat_n(B)\) that is
  Murray--von Neumann equivalent to~\(p_1\).  Thus \(p\cdot B^n
  \cong \mathcal{F}\).
\end{proof}

Fix a projection \(p\in \Mat_n(B)\) as in
Proposition~\ref{pro:cocompact_Hilbert_module}.  Using the
unitary~\(U\) in~\eqref{eq:csi_unitary}, we identify
\[
A \cong \mathcal{F}\otimes_B L^2(G,A)
\cong p\cdot B^n\otimes_B L^2(G,A)
\cong p\cdot L^2(G,A^n)
\]
for the representation of \(\Mat_n(B)\) on~\(L^2(G,A^n)\) induced by
the standard representation of~\(B\) on~\(L^2(G,A)\) by
\(G\)\nb-equivariant operators.  Let \(\Cont(\eta G)\subseteq
\Contb(G)\) be the Higson compactification for the coarse structure
on~\(G\) in Example~\ref{exa:coarse_group}.  This is equal
to the compactification of~\(\Cont_0(G)\) described in
Example~\ref{exa:nc_compactify_group} by
Lemma~\ref{lem:Higson_group}.  Let~\(M\) denote the representation
of~\(\Cont(\eta G)\) on \(L^2(G,A)\otimes \C^n \cong L^2(G,A^n)\) by
pointwise multiplication in the \(L^2G\)\nb-direction.  Let
\[
\varphi\colon \Cont(\eta G) \to \Bound(A),\qquad
f\mapsto p M(f) p,
\]
the compression of the representation by pointwise multiplication to
the direct summand \(p\cdot L^2(G,A^n) \cong A\).

\begin{theorem}
  \label{the:cocompact_csi_coarse}
  The subspace \(\cmp{A} \defeq A + \varphi(\Cont(\eta G)) \subseteq
  \Mult(A)\) is a compactification of~\(A\) with boundary
  \(\cmp{A}/A \cong \Cont(\bdy{\eta G})\).
\end{theorem}

\begin{proof}
  Consider the dense subalgebras \(\Contc(G)\subseteq \Cont_0(G)\)
  and \(\Contc(G)\subseteq \Cred(G)\).  When we represent both
  on~\(L^2(G)\),
  then \(M(\Contc(G)) \cdot \varrho(\Contc(G))\) is dense in the
  algebra of compactly supported integral kernels.  Hence
  \(\Cont_0(G)\cdot \Cred(G) = \Comp(L^2 G)\).  Similarly, when we
  represent \(\Cont_0(G)\) and \(A\rtimes_{\mathrm{r},\alpha} G\) on
  \(L^2(G,A)\) as above, then
  \[
  \Cont_0(G) \cdot (A\rtimes_{\mathrm{r},\alpha} G) =
  \Comp(L^2(G,A)).
  \]
  This is equivalent to the Imai--Takai Duality Theorem for crossed
  products for group actions and group coactions, which asserts an
  isomorphism
  \((A\rtimes_\mathrm{r} G) \rtimes \hat{G} \cong A\otimes \Comp(L^2
  G)\)
  (see \cite{Imai-Takai:Duality} or
  \cite{Echterhoff-Kaliszewski-Quigg-Raeburn:Categorical}*{Theorem~A.69}).

  The images of \(\Cont(\eta G)\)
  and~\(\Mat_n(A)\)
  in \(\Bound(L^2(G,A^n))\)
  commute.  And \(u_g f u_g^* - f = \lambda_g(f)-f\in \Cont_0(G)\) for
  all \(f\in\Cont(\eta G)\), \(g\in G\) by the definition of~\(\eta
  G\).  Therefore, \([B,\Cont(\eta G)] \subseteq \Cont_0(G) \cdot B
  = \Comp(L^2(G,A))\) as well.  So the map~\(\varphi\) above is
  multiplicative modulo \(\Comp(A)=A\).  Since \(\varphi(x^*)
  =\varphi(x)^*\), the induced map \(\Cont(\eta G) \to \Mult(A)/A\)
  is a \Star{}homomorphism.  Thus its image is a
  \(\Cst\)\nb-subalgebra of \(\Mult(A)/A\).  So
  \(A+\varphi(\Cont(\eta G))\) is a \(\Cst\)\nb-subalgebra
  of~\(\Mult(A)\).  It is unital because \(\varphi(1)=p\) is the
  identity operator on \(A \cong p \cdot L^2(G,A^n)\).  The argument
  above also shows that \(\varphi(\Cont_0(G)) \subseteq A\).  Thus
  the boundary \(\cmp{A} / A\) is a quotient of \(\Cont(\eta
  G)/\Cont_0(G) = \Cont(\bdy{\eta G})\).  To show that the induced
  map \(\Cont(\eta G)/\Cont_0(G) \to \cmp{A} / A\) is an
  isomorphism, we sketch the construction of a map \(\chi\colon
  \cmp{A} \to \Cont(\eta G)\) with \(\chi(A)\subseteq \Cont_0(G)\)
  and so that \(\chi\circ\varphi\) induces the identity map on
  \(\Cont(\eta G)/\Cont_0(G)\).

  Let \(w\in \Contc(G)\) be a positive function supported in a small
  neighbourhood of~\(1\) with \(\int_G w^2(g) \,\mathrm{d}g = 1\).
  Let \(w^g(x) \defeq w(g^{-1} x)\) for \(g,x\in G\).  We map
  \[
  \chi_1^g\colon
  \Bound(A) \cong \Bound(p\cdot L^2(G,A^n)) \to \Bound(L^2(G,A^n))
  \xrightarrow{C_{w^g}} \Bound(A^n) = \Mat_n(\Mult(A)),
  \]
  where the first arrow extends an operator by~\(0\) on the orthogonal
  complement and~\(C_{w^g}\) takes the matrix coefficient of \(w^g\in
  L^2(G)\), that is, \(\braket{C_{w^g}(x)\xi}{\eta} =
  \braket{x(\xi\otimes w^g)}{\eta\otimes w^g}\) for all
  \(\xi,\eta\in A^n\), \(x\in \Bound(L^2(G,A^n))\), \(g\in G\).  The
  operator \(\chi_1^1(1) = C_w(p)\) is positive.  If this is~\(0\),
  then~\(p\) vanishes on \(w\otimes A^n\) and hence on \(g\cdot w
  \otimes A^n\) for all \(g\in G\).  Since the \(G\)\nb-orbits of
  functions of the form~\(w\) are dense in~\(L^2 G\) and \(p\neq0\),
  there must be \(w\in\Contc(G)\) as above with \(C_w(p)\neq0\).
  Then we can find a positive linear functional \(\sigma\in A'\)
  with \(\sigma(C_w(p))=1\) and \(\norm{\sigma} \norm{C_w(p)}=1\).

  Recall that \(p\in \Mat_n(A)\rtimes_\mathrm{r} G\) acts on
  \(L^2(G,A^n)\) by a \(G\)\nb-equivariant operator, that is, \(p\)
  commutes with the operators
  \[
  U_g\colon L^2(G,A^n) \to L^2(G,A^n),\qquad
  (U_g f)(x) \defeq \alpha_g(f(g^{-1} x)).
  \]
  Thus \(C_w(p) = C_w(U_g^{-1} p U_g) =
  \alpha_g(C_{w^g}(p))\) for all \(g\in G\).  We define
  \[
  \chi\colon \Mult(A) = \Bound(A) \to \C^G,\qquad
  \chi(x)(g) \defeq \sigma\circ\alpha_g(\chi_1^g(x)).
  \]
  The map~\(\chi\) is completely positive by construction, and
  \[
  \chi(1)(g)
  = \sigma\alpha_g(C_{w^g}(p))
  = \sigma(C_w(p))
  = 1
  \]
  for all \(g\in G\).  Thus the function~\(\chi(x)\) on~\(G\) is
  bounded for all \(x\in \Mult(A)\) and \(\norm{\chi}=1\) as a map
  \(\Mult(A)\to\C^G\).

  If \(x\in A=\Comp(A)\),
  then its image in~\(\Bound(L^2(G,A^n))\)
  is compact, so we may approximate it by compactly supported integral
  kernels \(G\times G\to \Mat_n(A)\).
  For such an integral kernel~\(k\),
  \(g\mapsto \sigma(C_{w^g}(k))\)
  is continuous of compact support.  Hence
  \(\chi(A)\subseteq \Cont_0(G)\)
  as desired.  Let \(f\in\Cont(\eta G)\).  Then
  \begin{multline*}
    \chi(\varphi(f))(g)
    = \sigma\alpha_g(C_{w^g}(pM(f)p))
    \\= \sigma(C_w(U_g^{-1} pM(f)p U_g))
    = \sigma(C_w(p M(\lambda_{g^{-1}} f)p)).
  \end{multline*}
  Approximate~\(p\)
  by some \(\tilde{p}\in\Contc(G,\Mat_n(A))\)
  supported in a compact subset \(K\subseteq G\).
  Then \(C_w(\tilde{p} M(\lambda_{g^{-1}} f)\tilde{p})\)
  only involves the values of~\(\lambda_{g^{-1}} f\)
  on \(\supp w\cdot K\),
  that is, the values of~\(f\)
  on \(g\cdot \supp w \cdot K\).
  Since \(f\in\Cont(\eta G)\),
  \(f|_{g\cdot K\cdot \supp w}\)
  gets more and more constant.  Therefore,
  \[
  0 = \lim_{g\to\infty} \chi(\varphi(f))(g) - f(g)\sigma(C_w(p))
  = \lim_{g\to\infty} \chi(\varphi(f))(g) - f(g).
  \]
  Furthermore, since the \(G\)\nb-action on \(\Cont(\eta G)\) is
  continuous, \(\chi(\varphi(f))\) is a continuous function
  on~\(G\).  Thus \(\chi(\varphi(f)) - f \in \Cont_0(G)\).  So
  \(\chi(\cmp{A}) \subseteq \Cont(\eta G)\) and
  \(\chi\circ\varphi\) induces the identity map on \(\Cont(\eta
  G)/\Cont_0(G)\).
\end{proof}

\section{Noncommutative coarse maps and equivalences}
\label{sec:ncg_coarse_map}

We have encountered two situations that smell of a coarse equivalence
between two noncommutative spaces.  The first one is Rieffel
deformation for an action of an Abelian locally compact group~\(G\)
by translations, where we expect a coarse equivalence between
\(A\idealin \cmp{A}\)
and \(A^\Psi\idealin \cmp{A}{}^\Psi\).
The second is the coarse structure for a cocompact continuously
square-integrable action in Theorem~\ref{the:cocompact_csi_coarse},
which we expect to be coarsely equivalent to~\(\Cont_0(G)\)
with its usual coarse structure.  How should we define coarse maps
between noncommutative coarse spaces, so as to cover these two
situations?  In both cases, the corona algebras are isomorphic.  But a
\Star{}homomorphism between the boundaries is a poor definition of a
coarse map.

The example of the Moyal plane as a Rieffel deformation of
\(\Cont_0(\R^{2n})\)
shows that we cannot expect a \Star{}homomorphism between \(A\)
and~\(A^\Psi\).
This is not surprising because already in the commutative case, we
need discontinuous coarse maps, say, for the coarse equivalence
between \(\R\) and~\(\Z\).
An obvious way to allow for ``discontinuous'' \Star{}homomorphisms
would be to replace a \(\Cst\)\nb-algebra by its bidual
\(\Wst\)\nb-algebra.  This contains the \(\Cst\)\nb-algebra of Borel
functions on~\(X\) for \(\Cont_0(X)\), so that a normal
\Star{}homomorphism between the biduals would allow for a Borel map
between the locally compact spaces.  This idea fails, however, for
the Moyal plane: its underlying \(\Cst\)\nb-algebra is that of
compact operators, so its bidual is the
\(\Wst\)\nb-algebra~\(\Bound(\Hils)\) of bounded operators on a
Hilbert space.  There are no normal \Star{}homomorphisms from this
to a commutative \(\Wst\)\nb-algebra.

Rieffel's definition of his deformation is based on an isomorphism
between the Fr\'echet subspaces of smooth elements in \(A\)
and~\(A^\Psi\)
for the actions of~\(G\).
Once again, such a densely defined, unbounded map seems a poor way to
define noncommutative coarse maps.  Another ``quantisation map''
between \(A\)
and~\(A^\Psi\)
constructed in~\cite{Kaschek-Neumaier-Waldmann:Positivity} has much
better properties.  These were the motivation for the following
definition of a noncommutative coarse map.  As we shall see, the
coarse structure for a cocompact continuously square-integrable group
action in Theorem~\ref{the:cocompact_csi_coarse} also comes with
noncommutative coarse maps in this sense by construction.

\begin{definition}
  \label{def:coarse_map}
  A \emph{noncommutative coarse map} between two noncommutative coarse
  spaces \(A\idealin\cmp{A}\) and \(B\idealin\cmp{B}\) is a commuting
  diagram of maps
  \[
  \begin{tikzpicture}[baseline=(current bounding box.west)]
    \matrix (m) [cd,column sep=1em] {
      A & \cmp{A} & \bdy{A} \\
      B & \cmp{B} & \bdy{B}, \\
    };
    \draw[>->] (m-1-1) -- (m-1-2);
    \draw[>->] (m-2-1) -- (m-2-2);
    \draw[->>] (m-1-2) -- (m-1-3);
    \draw[->>] (m-2-2) -- (m-2-3);
    \draw[cdar] (m-1-1) -- node {\(\varphi\)} (m-2-1);
    \draw[cdar] (m-1-2) -- node {\(\cmp{\varphi}\)} (m-2-2);
    \draw[cdar] (m-1-3) -- node {\(\bdy{\varphi}\)} (m-2-3);
  \end{tikzpicture}
  \]
  with the following properties:
  \begin{enumerate}
  \item \(\cmp{\varphi}\) is a unital, completely positive map that is
    strictly continuous, that is, a continuous map if both \(\cmp{A}\)
    and~\(\cmp{B}\) carry the strict topology;
  \item \(\bdy{\varphi}\) is a \Star{}homomorphism.
  \end{enumerate}
\end{definition}

Since~\(\cmp{\varphi}\) is strictly continuous and completely
positive, so is~\(\varphi\).  Even more, if \((u_i)_{i\in I}\) is an
approximate unit in~\(A\), then the net~\((u_i)_{i\in I}\) converges
strictly to \(1\in\cmp{A}\) and hence \(\varphi(u_i)_{i\in I}\)
converges strictly to \(\cmp{\varphi}(1) = 1\) in~\(\Mult(B)\).
Thus~\(\varphi\) is a \emph{nondegenerate completely positive map}.
Conversely, let \(\varphi\colon A\to B\) be a nondegenerate
completely positive map; that is, \(\varphi\) is completely positive
and the net \(\varphi(u_i)_{i\in I}\) converges strictly to~\(1\)
if~\((u_i)_{i\in I}\) is an approximate unit in~\(A\).
Then~\(\varphi\) extends uniquely to a strictly continuous, unital,
completely positive map \(\Mult(\varphi)\colon
\Mult(A)\to\Mult(B)\); this follows, say, from the Stinespring
Dilation Theorem for nondegenerate completely positive maps, see
\cite{Lance:Hilbert_modules}*{Corollary 5.7}.  (Lance only asserts
strict continuity on the unit balls.  The Cohen--Hewitt
Factorisation Theorem allows to prove strict continuity everywhere.)

The arguments above show that a noncommutative coarse map is already
determined by the map \(\varphi\colon A\to B\).  This is a
nondegenerate, completely positive map \(\varphi\colon A\to B\) with
the following two extra properties:
\begin{itemize}
\item \(\Mult(\varphi)\colon \Mult(A)\to\Mult(B)\) maps
  \(\cmp{A}\subseteq \Mult(A)\) to \(\cmp{B}\subseteq \Mult(B)\),
  giving \(\cmp{\varphi}\colon \cmp{A}\to\cmp{B}\);
\item \(\cmp{\varphi}(a_1 a_2) - \cmp{\varphi}(a_1)
  \cmp{\varphi}(a_2)\in B\) for all \(a_1,a_2\in\cmp{A}\).
\end{itemize}
The second condition says that the induced map \(\bdy{\varphi}\colon
\bdy{A}\to\bdy{B}\) is multiplicative; this map exists because
\(\cmp{\varphi}(\cmp{A})\subseteq \cmp{B}\) and
\(\cmp{\varphi}(A)\subseteq B\).  Then~\(\bdy{\varphi}\) is a
\Star{}homomorphism because, as a completely positive map, it is
linear and preserves adjoints.  Any nondegenerate, completely
positive map \(\varphi\colon A\to B\) with these two properties
gives a noncommutative coarse map.

To define coarse equivalences, we must also define when to
noncommutative coarse maps are ``close.''  In the commutative case,
closeness may be defined for maps from any set into a coarse space.
We define closeness in similar generality:

\begin{definition}
  \label{def:close}
  Let \(A\idealin \cmp{A}\) be a noncommutative coarse space and
  let~\(B\) be a \(\Cst\)\nb-algebra.  Two nondegenerate completely
  positive maps \(\varphi,\psi\colon A\rightrightarrows B\) are
  \emph{close} if their strictly continuous extensions
  \(\Mult(\varphi),\Mult(\psi)\colon \Mult(A)\rightrightarrows
  \Mult(B)\) satisfy \(\bigl(\Mult(\varphi)-\Mult(\psi)\bigr)
  (\cmp{A}) \subseteq B\).  Thus two noncommutative coarse maps
  \((\varphi,\cmp{\varphi},\bdy{\varphi})\) and
  \((\psi,\cmp{\psi},\bdy{\psi})\) from \(A\idealin \cmp{A}\) to
  \(B\idealin \cmp{B}\) are close if and only if
  \(\bdy{\varphi}=\bdy{\psi}\).
\end{definition}

Noncommutative coarse maps may be composed in an obvious fashion.
So they define a category of noncommutative coarse spaces.  This
composition respects the closeness relation.  So equivalence classes
of noncommutative coarse maps up to closeness still form a category.
Taking the boundary quotient is a functor in this quotient category.
A noncommutative coarse equivalence is defined as an isomorphism in
this quotient category.  More explicitly:

\begin{definition}
  \label{def:coarse_equivalence}
  A noncommutative coarse map \(\varphi\) is called a \emph{coarse
    equivalence} if there is a noncommutative coarse map \(\psi\) in
  the opposite direction such that \(\varphi\circ\psi\) and
  \(\psi\circ\varphi\) are close to the identity maps.
\end{definition}

\begin{example}
  Let~\(A\) carry a cocompact continuously square-integrable action
  of a locally compact group~\(G\).  The maps \(\varphi\colon
  \Cont_0(G) \to A\) and \(\chi\colon A\to \Cont_0(G)\) constructed
  before Theorem~\ref{the:cocompact_csi_coarse} and during its proof
  are noncommutative coarse maps that are inverse to each other up
  to closeness.  So they form a coarse equivalence between \(A\)
  and~\(\Cont_0(G)\).  This is expected because in the commutative
  case, the unique \(G\)\nb-invariant coarse structure on a
  cocompact proper \(G\)\nb-space is coarsely equivalent to~\(G\).
  We have already constructed \(\varphi\) and~\(\chi\) as unital,
  completely positive maps between \(\cmp{A}\) and~\(\Cont(\eta
  G)\).  It is routine to check that \(\varphi\) and~\(\chi\) are
  strictly continuous, using the known fact that all states on~\(A\)
  are strictly continuous.  We checked during the proof of
  Theorem~\ref{the:cocompact_csi_coarse} that~\(\varphi\) induces a
  \Star{}homomorphism between the boundaries; this is an isomorphism
  by construction of~\(\cmp{A}\).  And we checked that~\(\chi\)
  induces the inverse map~\(\varphi^{-1}\) between the boundaries,
  so~\(\bdy{\chi}\) is a \Star{}homomorphism as well.
\end{example}

To further justify our definition of a noncommutative coarse map, we
now construct them in the context of Rieffel deformations.  So
let~\(G\) be a locally compact Abelian group and let~\(\Psi\) be a
\(2\)\nb-cocycle on~\(G\).  Let \(A\idealin \cmp{A}\) be a
noncommutative coarse space with a \(G\)\nb-action~\(\alpha\) by
translations.  We want to construct a noncommutative coarse map
\((\varphi,\cmp{\varphi},\bdy{\varphi})\) from \(A\idealin \cmp{A}\)
to \(A^\Psi\idealin \cmp{A}{}^\Psi\) such that \(\bdy{\varphi}\) is
the canonical isomorphism \(\bdy{A} \cong \bdy{A}^\Psi\) in
Theorem~\ref{the:Rieffel_deform_translation_action}.

The following construction works for an arbitrary
\(G\)\nb-\(\Cst\)-algebra~\(A\), so it also applies to \(\cmp{A}\)
and~\(\bdy{A}\).  We need Kasprzak's description of the Rieffel
deformation~\(A^\Psi\) and recall some of his notation,
see~\cite{Kasprzak:Rieffel_deformation}.  Let \(B\defeq
A\rtimes_\alpha G\) and let \(\hat\alpha\colon \hat{G}\to\Aut(B)\) be
the dual action on the crossed product.  Let \(\lambda\colon
\Cred(G)\to \Mult(B)\) be the standard inclusion.  The triple
\((B,\hat\alpha,\lambda)\) is called a \emph{\(G\)\nb-product}.  There
is also a canonical embedding \(A\subseteq \Mult(B)\), and its image
may be described using the \(G\)\nb-product structure: it consists of
those elements of~\(\Mult(B)\) that satisfy \emph{Landstad's
  conditions}.  In order to deform~\(A\), Kasprzak deforms the dual
action~\(\hat\alpha\) using the \(2\)\nb-cocycle~\(\Psi\).  The new
dual action~\(\hat\alpha^\Psi\) is still part of a \(G\)\nb-product
\((B,\hat\alpha^\Psi,\lambda)\), and~\(A^\Psi\) is the resulting
Landstad algebra.
Let \(b\in B\) and \(f_1,f_2\in \Cred(G)\cap L^2(G)\).  Then
\[
E^\Psi(f_1 b f_2) \defeq
\int_{\hat{G}} \hat\alpha^\Psi_\gamma(f_1 b f_2) \,\textup{d}\gamma
\]
belongs to~\(A^\Psi\), and \(\norm{E^\Psi(f_1 b f_2)} \le
\norm{f_1}_2 \norm{b} \norm{f_2}_2\); here the integral converges in
the strict topology on~\(\Mult(B)\).  Furthermore, Kasprzak shows
that elements of this form are dense in~\(A^\Psi\).

We want to use the same formula for \(b\in A\subseteq \Mult(B)\),
so~\(b\) no longer belongs to~\(B\).  But then \(f\cdot b\in B\) and
\(b\cdot f\in B\) for any \(f\in\Cred(G)\).  The \(\Cred(G)\)-module
\(\Cred(G)\cap L^2(G)\) with the norm \(f\mapsto \norm{f}_{\Cred(G)}
+ \norm{f}_2\) is nondegenerate because \(\Contc(G)\) is dense.  By
the Cohen--Hewitt Factorisation Theorem, any \(f_1\in \Cred(G)\cap
L^2(G)\) may also be written as \(f_1=f_1'\cdot f_1''\) with
\(f_1'\in \Cred(G)\cap L^2(G)\), \(f_1''\in\Cred(G)\), and
\[
\norm{f_1'}_{\Cred(G)\cap L^2(G)} \norm{f_1''}_{\Cred(G)} \le
\norm{f_1}_{\Cred(G)\cap L^2(G)} + \varepsilon
\]
for any \(\varepsilon>0\).
Since \(\Cred(G)\cdot A\subseteq B\),
we get \(E^\Psi(f_1 a f_2)\in A^\Psi\) and
\[
\norm{E^\Psi(f_1 a f_2)} \le \norm{f_1}_2 \norm{a} \norm{f_2}_2
\]
if \(f_1,f_2\in \Cred(G)\cap L^2(G)\), \(a\in A\).

\begin{theorem}
  \label{the:nc_map_Rieffel}
  Let~\(G\)
  be a locally compact Abelian group with a \(2\)\nb-cocycle~\(\Psi\).
  Let \(A\idealin \cmp{A}\)
  be a noncommutative coarse space with a continuous \(G\)\nb-action
  by translations.  Let \(f\in L^2(G)\cap \Cred(G)\)
  satisfy \(\norm{f}_2=1\).
  Then \(\cmp{\varphi}(a) \defeq E^\Psi(f a f^*)\)
  for \(a\in\cmp{A}\)
  and the same formula for \(\varphi = \cmp{\varphi}|_A\)
  and \(\bdy{\varphi}\)
  define a noncommutative coarse map from \(A\idealin \cmp{A}\)
  to \(A^\Psi\idealin \cmp{A}{}^\Psi\).
  The map \(\bdy{\varphi}\)
  is the canonical isomorphism \(\bdy{A} = \bdy{A}^\Psi\)
  for any~\(f\).  All these maps for different~\(f\) are close.
\end{theorem}

\begin{proof}
  The naturality of the construction of~\(\varphi\)
  gives a commuting diagram as in Definition~\ref{def:coarse_map}.  A
  map of the
  form \(a\mapsto f a f^*\) is always completely positive.  Since
  automorphisms are \Star{}homomorphisms, they are completely
  positive, and so is an integral over a family of completely
  positive maps.  Thus \(a\mapsto \int_{K} \hat\alpha^\Psi_\gamma(f
  a f^*)\,\textup{d}\gamma\) is a completely positive map \(A\mapsto
  \Mult(B)\) for any compact subset \(K\subseteq \hat{G}\).
  Finally, we let \(K\to\hat{G}\) and take a strict limit.  This
  still gives a completely positive map.  So \(a\mapsto E^\Psi(f a
  f^*)\) is a completely positive map \(A\to A^\Psi\) for any \(f\in
  \Cred(G)\cap L^2(G)\).  If \(a_2\in A^\Psi\), then
  \(\hat\alpha^\Psi_\gamma(a_2)=a_2\) for all \(\gamma\in \hat{G}\).  So
  \[
  a_2\cdot E^\Psi(f a f^*)
  = E^\Psi(a_2\cdot f a f^*),\qquad
  E^\Psi(f a f^*)\cdot a_2
  = E^\Psi(f a f^*\cdot a_2).
  \]
  The products \(a_2 f\) and~\(f^* a_2\) belong to \(\Cred(G,A)\cap
  L^2(G,A)\), on which~\(A\) acts nondegenerately.  So \(a_2 f a\)
  and \(a f^* a_2\) go to~\(0\) in the norm of \(\Cred(G,A)\cap
  L^2(G,A)\) if \(a\) goes to~\(0\) strictly.  This suffices for
  \(a_2\cdot E^\Psi(f a f^*)\) and \(E^\Psi(f a f^*)\cdot a_2\) to
  go to~\(0\) in norm if \(a\) goes to~\(0\) strictly.  That is, the
  map \(a\mapsto E^\Psi(f a f^*)\) is strictly continuous.

  Since the \(G\)\nb-action on~\(\bdy{A}\) is trivial, \(\bdy{A}\)
  and~\(\Cred(G)\) commute in~\(B\); so
  \[
  E^\Psi(f a f^*) = a\cdot
  \int_{\hat{G}} \hat\alpha^\Psi_\gamma(f* f^*) \,\textup{d}\gamma
  \]
  in \(\bdy{A}^\Psi\) for all \(a\in \bdy{A}\).  The integral above
  is a constant multiple of the identity element, where the constant
  is \(\bigl| \hat{f}(1) \bigr|{}^2 = \norm{f}_2^2 = 1\) by the
  normalisation assumption.  So the map \(a\mapsto E^\Psi(f a f^*)\)
  is the identity map from \(\bdy{A}\subseteq\Mult(B)\) to
  \(\bdy{A}^\Psi\subseteq \Mult(B)\), which is a
  \Star{}homomorphism.  Since it does not depend on~\(f\), the maps
  for different~\(f\) are all close to each other.  The same
  computation ensures that \(\cmp{\varphi}(1)=1\).
\end{proof}

\begin{proposition}
  \label{pro:Rieffel_coarse_equivalence}
  Let \(A\idealin \cmp{A}\)
  be a noncommutative coarse space with a continuous \(G\)\nb-action
  by translations.  Let \(\Psi_1,\Psi_2\)
  be continuous \(2\)\nb-cocycles
  on~\(G\).  Define noncommutative coarse maps
  \begin{align*}
    (\varphi_{01},\cmp{\varphi}_{01},\bdy{\varphi}_{01})&\colon
    (A\idealin \cmp{A}) \to (A^{\Psi_1}\idealin \cmp{A}{}^{\Psi_1}),\\
    (\varphi_{12},\cmp{\varphi}_{12},\bdy{\varphi}_{12})&\colon
    (A^{\Psi_1}\idealin \cmp{A}{}^{\Psi_1})
    \to
    \bigl((A^{\Psi_1})^{\Psi_2}\idealin
    (\cmp{A}{}^{\Psi_1})^{\Psi_2}\bigr),\\
    (\varphi_{02},\cmp{\varphi}_{02},\bdy{\varphi}_{02})&\colon
    (A\idealin \cmp{A}) \to (A^{\Psi_1+\Psi_2}\idealin
    \cmp{A}{}^{\Psi_1+\Psi_2}).
  \end{align*}
  Identify \(A^{\Psi_1+\Psi_2} \cong (A^{\Psi_1})^{\Psi_2}\) and
  \((\cmp{A}{}^{\Psi_1})^{\Psi_2} \cong \cmp{A}{}^{\Psi_1+\Psi_2}\) as
  in~\ref{enum:RD4}.  Then \(\varphi_{02}\) is close to
  \(\varphi_{12}\circ\varphi_{01}\).
\end{proposition}

\begin{proof}
  Both \(\varphi_{02}\) and \(\varphi_{12}\circ\varphi_{01}\) induce
  the canonical isomorphism on the corona algebras, hence they are
  close.
\end{proof}

\begin{theorem}
  The coarse map from \(A\idealin \cmp{A}\) to \(A^\Psi\idealin
  \cmp{A}{}^\Psi\) constructed in
  Theorem~\textup{\ref{the:nc_map_Rieffel}} is a coarse
  equivalence.
\end{theorem}

\begin{proof}
  Apply Proposition~\ref{pro:Rieffel_coarse_equivalence} with
  \(\Psi_1=\Psi\) and \(\Psi_2=\Psi^*\) and use that the
  noncommutative coarse maps from \(A\idealin \cmp{A}\) to itself
  constructed in Theorem~\ref{the:nc_map_Rieffel} are close to
  the identity map because they induce the identity map on the
  corona algebras.
\end{proof}

\section{Coarse maps: the commutative case}
\label{sec:coarse_maps_commutative}

How are our noncommutative coarse maps related to coarse maps between
ordinary coarse spaces?  As we shall see, ordinary coarse maps give
noncommutative coarse maps.  The converse holds for coarse structures
from compactifications with second countable boundary.  It is unclear,
however, for coarse structures defined by metrics.  For commutative
\(\Cst\)\nb-algebras,
we are working with compactifications and not with coarse structures.
Hence we first have to discuss two categories of spaces with different
extra structure.

\begin{definition}
  \label{def:coarse_categories}
  Let~$\Coarse$ be the category of proper coarse spaces with coarse
  maps, that is, maps that are proper and preserve controlled subsets.
  Let~$\Compactified$ be the category of compactified spaces
  $X \subseteq \cmp{X}$, where $X$ is a \(\sigma\)\nb-compact,
  locally compact space with a compactification~$\cmp{X}$; the
  morphisms in~$\Compactified$ from \(X\subseteq \cmp{X}\)
  to \(Y\subseteq \cmp{Y}\)
  are maps \(\cmp{X}\to\cmp{Y}\)
  that map \(X\)
  into~\(Y\)
  and \(\bdy{X}\)
  into~\(\bdy{Y}\),
  and that are continuous on~\(\bdy{X}\).
  Two morphisms in~\(\Compactified\)
  are called \emph{close} if they induce the same map
  \(\bdy{X}\to\bdy{Y}\).
\end{definition}

Let \(X\subseteq \cmp{X}\)
and \(Y\subseteq \cmp{Y}\)
be objects of~\(\Compactified\)
and let \(\cmp{F}\colon \cmp{X}\to\cmp{Y}\)
be a morphism between them.  Then~\(F\)
restricts to a map \(F\colon X\to Y\)
and to a continuous map \(\bdy{F}\colon \bdy{X}\to\bdy{Y}\).
We claim that~\(F\)
determines~\(\cmp{F}\) uniquely.
Let \(x\in \bdy{X}\).
Since~\(X\)
is dense in~\(\cmp{X}\),
there is a net \((x_i)_{i\in I}\)
in~\(X\)
that converges towards~\(x\)
in~\(\cmp{X}\).
Since~\(\cmp{F}\)
is continuous at~\(x\), \(\cmp{F}(x) = \lim F(x_i)\).

Thus we may also define the morphisms in~\(\Compactified\)
as maps \(F\colon X\to Y\)
that have an extension \(\cmp{F}\colon \cmp{X}\to\cmp{Y}\)
that is continuous on~\(\bdy{X}\)
and that maps~\(\bdy{X}\)
to~\(\bdy{Y}\).
The latter condition means that~\(F\)
is proper, that is, preimages of relatively compact subsets of~\(Y\)
are relatively compact in~\(X\).
The map~\(F\)
need not be continuous, just as coarse maps are not required to be
continuous.

\begin{proposition}[\cite{Roe:Lectures}*{Proposition~2.41}]
  \label{pro:Higson_functor_and_back}
  The Higson compactification is part of a functor
  \(H\colon \Coarse\to\Compactified\),
  that is, a coarse map \(X\to Y\)
  has an extension \(\cmp{X}\to\cmp{Y}\)
  that is continuous on~\(\bdy{X}\)
  and maps~\(\bdy{X}\)
  to~\(\bdy{Y}\).  This functor preserves closeness of morphisms.
\end{proposition}

Conversely, a compactification~\(\cmp{X}\) of a
\(\sigma\)\nb-compact, locally compact space~\(X\) induces a coarse
structure~\(T(\cmp{X})\) on~\(X\).  Let \(\Compactified_p\subseteq
\Compactified\) be the full subcategory consisting of those
compactified spaces where the subset
in~\eqref{eq:bad_points_for_control} and \(\{(x,x)\mid x\in X\}\) in
\(\cmp{X}\times\cmp{X}\) are separated by neighbourhoods.  This is
equivalent to \(T(\cmp{X})\) being proper, see
Proposition~\ref{pro:compactification2coarse_metrisable}.  It is
also observed there that \(H(\Coarse)\subseteq \Compactified_p\).

\begin{lemma}
  \label{lem:functoriality}
  A morphism~\(\cmp{F}\)
  from \(X\subseteq \cmp{X}\)
  to \(Y\subseteq \cmp{Y}\)
  in~$\Compactified_p$ restricts to a coarse map \(X\to Y\)
  with respect to the coarse structures \(T(\cmp{X})\)
  and~\(T(\cmp{Y})\).
  Thus we get a functor \(T\colon \Compactified_p\to\Coarse\).
  It preserves closeness of morphisms.
\end{lemma}

\begin{proof}
  Continuity of~\(\cmp{F}\)
  at \(x\in\bdy{X}\)
  means that \(\cmp{F}^{-1}(U)\)
  is a neighbourhood of~\(x\)
  if~\(U\)
  is a neighbourhood of~\(\cmp{F}(x)\).
  Since~\(\cmp{F}\)
  is continuous on~\(\bdy{X}\),
  \(\cmp{F}^{-1}(U)\)
  is a neighbourhood of~\(\bdy{X}\)
  if~\(U\)
  is a neighbourhood of \(\bdy{Y}\supseteq \cmp{F}(\bdy{X})\).
  This means that \(F\defeq \cmp{F}|_X\colon X\to Y\) is proper.

  Let \(E\subseteq X\times X\)
  be controlled in~\(T(\cmp{X})\).
  Equivalently, if \((x_i)_{i\in I}\)
  and~\((y_i)_{i\in I}\)
  are nets in~\(X\)
  that converge in~\(\cmp{X}\)
  with \(\lim x_i\in \bdy{X}\) or \(\lim y_i\in \bdy{X}\)
  and that satisfy \((x_i,y_i)\in E\)
  for all \(i\in I\),
  then \(\lim x_i=\lim y_i\).
  Any net \((x_i,y_i)_{i\in I}\)
  in \((F\times F)(E)\)
  lifts to a net \((\hat{x}_i,\hat{y}_i)_{i\in I}\)
  in~\(E\).
  Since~\(E\)
  is controlled, accumulation points of
  \((\hat{x}_i,\hat{y}_i)_{i\in I}\)
  belong to \(X\times X\)
  or are of the form~\((x,x)\)
  with \(x\in \bdy{X}\).
  Hence all accumulation points of the net \((x_i,y_i)_{i\in I}\)
  in \((F\times F)(E)\)
  belong to \(Y\times Y\supseteq F(X\times X)\)
  or are of the form~\((y,y)\)
  with \(y\in \bdy{Y}\).
  Thus \((F\times F)(E)\)
  is controlled in~\(T(\cmp{Y})\).  So~\(F\) is a coarse map.

  Let \(\cmp{F_1}\)
  and~\(\cmp{F_2}\)
  be close morphisms, that is, \(\bdy{F_1} = \bdy{F_2}\).
  If \(F_1\)
  and~\(F_2\)
  are not close with respect to the coarse structure~\(T(\cmp{Y})\)
  on~\(Y\),
  then \(\{(F_1(x),F_2(x)) \mid x\in X\}\)
  is not controlled.  Hence there is a net \((x_i)_{i\in I}\)
  in~\(X\)
  for which \(\lim F_1(x_i)\)
  and \(\lim F_2(x_i)\)
  exist and are different and one limit belongs to~\(\bdy{Y}\).
  Since~\(\cmp{X}\)
  is compact, we may pass to a subnet of~\((x_i)_{i\in I}\)
  that converges in~\(\cmp{X}\).
  Since \(F_1\)
  and~\(F_2\)
  are proper, one limit must belong to~\(\bdy{X}\).
  Then
  \(\lim F_1(x_i) = \bdy{F_1}(\lim x_i) = \bdy{F_2}(\lim x_i) = \lim
  F_2(x_i)\), contradiction.  Thus \(F_1\) and~\(F_2\) are close.
\end{proof}

\begin{corollary}
  \label{cor:coarse_map_metrisable_boundary}
  If the coarse structures on \(X\)
  and~\(Y\)
  come from proper metrics or from compactifications with second
  countable boundary, then a map \(F\colon X\to Y\)
  is coarse if and only if it is a morphism in~\(\Compactified_p\),
  and two coarse maps are close if and only if they are close as
  morphisms in~\(\Compactified_p\).
\end{corollary}

\begin{proof}
  Propositions \ref{pro:compactification2coarse_metrisable} and
  \ref{pro:metric_coarse_compactification_equivalent} say that the
  functors \(H\colon \Coarse\to\Compactified_p\)
  and \(T\colon \Compactified_p\to\Coarse\)
  satisfy \(H\circ T(X,\cmp{X}) = (X,\cmp{X})\)
  and \(T\circ H(X,d)=(X,d)\)
  if \(X\subseteq \cmp{X}\)
  is a compactification with second countable boundary or if~\((X,d)\)
  denotes~\(X\)
  with the coarse structure of a metric~\(d\)
  on~\(X\).
  Now use that both \(H\)
  and~\(T\) are functors that preserve closeness.
\end{proof}

What are noncommutative coarse maps in the commutative case?  If
\(\varphi\colon \Cont_0(X)\to\Cont_0(Y)\) is a nondegenerate
completely positive contraction, then so are the maps
\(\varphi_y\colon \Cont_0(X)\to\C\), \(h\mapsto \varphi(h)(y)\), for
\(y\in Y\).  So each~\(\varphi_y\) is a state on~\(\Cont_0(X)\).
All states are of the form \(h\mapsto \int_X
h(x)\,\textup{d}\mu_y(x)\) for a unique probability
measure~\(\mu_y\) on~\(X\).  Let \(\mathcal{P}(X)\) be the set of
probability measures on~\(X\) with the weak topology defined by the
pairing with \(\Cont_0(X)\).  The functions~\(\varphi(h)\) for
\(h\in\Cont_0(X)\) are continuous if and only if the map
\(Y\to\mathcal{P}(X)\), \(y\mapsto\mu_y\), is continuous.  And the
function~\(\varphi(h)\) vanishes at~\(\infty\) if and only if
\(\lim_{y\to\infty} \int_X h(x)\,\textup{d}\mu_y(x)=0\).  So
nondegenerate completely positive contractions
\(\Cont_0(X)\to\Cont_0(Y)\) are equivalent to continuous maps
\(\mu\colon Y\to\mathcal{P}(X)\) that vanish at~\(\infty\).

\begin{lemma}
  \label{lem:nc_coarse_map}
  A continuous map \(\mu\colon Y\to\mathcal{P}(X)\) that vanishes
  at~\(\infty\) corresponds to a noncommutative coarse map from
  \(\Cont_0(X)\idealin \Cont(\cmp{X})\) to \(\Cont_0(Y)\idealin
  \Cont(\cmp{Y})\) if and only if the following holds: if a
  net~\((y_i)_{i\in I}\) converges in~\(\cmp{Y}\) to some
  \(y_\infty\in\bdy{Y}\), then there is \(x_\infty\in \bdy{X}\) such
  that \(\lim_i \int_X h(x)\,\textup{d}\mu_{y_i}(x) = h(x_\infty)\)
  for all \(h\in\Cont(\cmp{X})\).  In brief,
  \[
  \lim_{y\to y_\infty}
  \int_X h(x)\,\textup{d}\mu_y(x) = h(x_\infty)
  \]
  for all \(h\in\Cont(\cmp{X})\).  Two
  continuous maps \(\mu_1,\mu_2\colon
  Y\rightrightarrows\mathcal{P}(X)\) that vanish at~\(\infty\) are
  close as noncommutative coarse maps from \(\Cont_0(X)\idealin
  \Cont(\cmp{X})\) to \(\Cont_0(Y)\) if and only if
  \begin{equation}
    \label{eq:close_nc_maps_measures}
    \lim_{y\to\infty} \int_X h(x)\,\textup{d}\mu_{1,y}(x)
    - \int_X h(x)\,\textup{d}\mu_{2,y}(x) = 0\qquad
    \text{for all }h\in\Cont(\cmp{X}).
  \end{equation}
\end{lemma}

\begin{proof}
  The condition in the lemma for a noncommutative coarse map is
  clearly necessary.  Conversely, assume this
  condition for all convergent nets~\((y_i)_{i\in I}\).  Merging two
  nets with the same limit \(y_\infty\in\bdy{Y}\), we see that they
  give the same \(x_\infty\in\bdy{X}\).  So \(y_\infty\mapsto
  x_\infty\) well-defines a map \(\bdy{\mu}\colon
  \bdy{Y}\to\bdy{X}\).  If \(h\in\Cont(\cmp{X})\), then we define
  \(\cmp{\varphi}(h)\colon \cmp{Y}\to\C\) by \(\cmp{\varphi}(h)(y) =
  \int_X h(x)\,\textup{d}\mu_y(x)\) for \(y\in Y\) and
  \(\cmp{\varphi}(h)(y) = h(\bdy{\mu}(y))\) for \(y\in\bdy{Y}\).
  Then \(\lim \cmp{\varphi}(h)(y_i) = \cmp{\varphi}(h)(\lim y_i)\)
  whenever~\((y_i)_{i\in I}\) is a net in~\(Y\) that converges
  in~\(\cmp{Y}\): if the limit lies in~\(Y\), this is the continuity
  of \(y\mapsto \mu_y\), and if the limit lies in~\(\bdy{Y}\), it is
  the construction of~\(\bdy{\mu}\).  Since~\(Y\) is dense
  in~\(\cmp{Y}\), this already implies the continuity
  of~\(\cmp{\varphi}(h)\) on~\(\cmp{Y}\) (see
  \cite{Banerjee:Thesis}*{Lemma~3.30}).
  So~\(\cmp{\varphi}\) is a noncommutative coarse map.
  Two maps \(\mu_1,\mu_2\colon \Cont_0(X)\rightrightarrows\Cont_0(Y)\)
  are close
  if and only if the strictly continuous extension of
  \(\mu_1-\mu_2\) maps \(\Cont(\cmp{X})\) to~\(\Cont_0(Y)\).  This
  is equivalent to~\eqref{eq:close_nc_maps_measures}.
\end{proof}

\begin{proposition}
  \label{pro:coarse_map_nc}
  Let \(X\subseteq \cmp{X}\)
  and \(Y\subseteq \cmp{Y}\)
  be objects of~\(\Compactified_p\)
  and let \(\cmp{F}\colon \cmp{X}\to\cmp{Y}\)
  be a morphism in~\(\Compactified_p\).
  There is a noncommutative coarse map~\(\varphi\)
  from \(\Cont_0(Y)\idealin \Cont(\cmp{Y})\)
  to \(\Cont_0(X)\idealin \Cont(\cmp{X})\)
  so that \(\bdy{\varphi} = \bdy{F}^*\).
  Two morphisms \(\cmp{F_1}\)
  and~\(\cmp{F_2}\)
  are close if and only if the corresponding noncommutative coarse
  maps are close.
\end{proposition}

\begin{proof}
  By assumption, there is a neighbourhood~\(U\)
  of the diagonal in \(X\times X\)
  that is controlled with respect to the coarse
  structure~\(T(\cmp{X})\)
  associated to~\(\cmp{X}\).
  The open subsets \(V\subseteq X\)
  with \(V\times V\subseteq U\)
  cover~\(X\).
  Since~\(X\)
  is locally compact and \(\sigma\)\nb-compact,
  it is paracompact.  So there is a partition of unity
  \(\sum_{i\in \N} \psi_i = 1\)
  with \((\supp \psi_i)^2\subseteq U\).
  Choose \(x_i\in \supp \psi_i\)
  and let \(y_i\defeq F(x_i)\)
  for \(i\in\N\).
  Define
  \[
  (\varphi h)(x) \defeq \sum_{i\in\N} h(y_i)\psi_i(x)
  \]
  for \(h\in\Cont_0(Y)\),
  \(x\in X\).
  We claim that~\(\varphi\)
  is a noncommutative coarse map.

  The map~\(\varphi\)
  is a completely positive contraction because \(\psi_i\ge0\)
  for all \(i\in\N\)
  and \(\sum \psi_i=1\).
  The value~\(\varphi(h)(x)\)
  is a convex combination of \(h(F(x_i))\)
  for those \(i\in\N\)
  with \(\psi_i(x)\neq0\).
  So \((x,x_i)\in(\supp\psi_i)^2 \subseteq U\).
  This allows to prove both that
  \(\varphi(\Cont_0(Y))\subseteq \Cont_0(X)\)
  and that \(\lim \varphi(u_i)=1\)
  if~\((u_i)\)
  is an approximate unit for~\(\Cont_0(Y)\).
  So~\(\varphi\)
  is a nondegenerate completely positive contraction
  \(\Cont_0(Y)\to\Cont_0(X)\).
  If \(h\in\Cont(\cmp{Y})\),
  then \(\abs{h(F(x))-h(F(x_i))}\to0\)
  as \(x\to\bdy{X}\)
  in~\(U\).
  Thus the function \(x\mapsto h\circ F(x) - \varphi(h)(x)\)
  on~\(X\)
  vanishes at~\(\infty\).
  Since the function~\(h\circ \cmp{F}\)
  on~\(\cmp{X}\)
  is continuous on~\(\bdy{X}\),
  where~\(\cmp{F}\)
  is continuous, it follows that \(\varphi(h)\)
  is continuous on~\(\bdy{X}\).
  Since continuity on~\(X\)
  is clear, we get
  \(\varphi(\Cont(\cmp{Y})) \subseteq \Cont(\cmp{X})\).
  Furthermore, \(\bdy{F}^* = \bdy{\varphi}\)
  as desired.  By definition, \(\cmp{F_1}\)
  and~\(\cmp{F_2}\)
  are close if and only if \(\bdy{F_1} = \bdy{F_2}\),
  if and only if \(\bdy{\varphi_1} = \bdy{\varphi_2}\),
  if and only if \(\varphi_1\) and~\(\varphi_2\) are close.
\end{proof}

It is unclear, in general, whether any noncommutative coarse map from
\(\Cont_0(Y)\idealin \Cont(\cmp{Y})\)
to \(\Cont_0(X)\idealin \Cont(\cmp{X})\)
comes from a morphism \(\cmp{F}\colon \cmp{X}\to\cmp{Y}\)
in~\(\Compactified_p\).
Theorem~\ref{the:lift_metrisable} will show this
if~\(\bdy{Y}\)
is second countable.  Even more, any unital \Star{}homomorphism
\(\Cont(\bdy{Y})\to\Cont(\bdy{X})\)
lifts to a morphism in~\(\Compactified_p\).
The authors tried without success to prove this for Higson
compactifications of metric coarse structures.
We mention partial results that we obtained in this direction.
The problem whether every noncommutative coarse map
\(\varphi\colon \Cont_0(Y)\to\Cont_0(X)\)
lifts to an ordinary coarse map is invariant under coarse
equivalence by Proposition~\ref{pro:coarse_map_nc}.
Any (proper) coarse space is coarsely equivalent to a discrete one.
This reduces the problem to the case where \(X\)
and~\(Y\)
are countable sets equipped with proper metrics.  If a noncommutative
coarse map does not lift to an ordinary coarse map, then this is
witnessed by evaluation at a sequence of points~\((x_n)_{n\in\N}\)
in~\(X\)
that goes to~\(\infty\).
Any subsequence of~\((x_n)_{n\in\N}\)
still witnesses the non-existence of a lifting.  We may pass to a
subsequence to arrange that \(d(x_n,x_m)>\abs{2^n-2^m}\)
for all \(n,m\in\N\).
Replacing~\(X\)
by the subset \(\{x_n\mid n\in\N\}\),
we get another counterexample where \(X=\N\)
with the discrete coarse structure.  This is the unique coarse
structure on~\(\N\)
where the Higson compactification is the Stone--\v{C}ech
compactification, that is, \(\Cont_0(X)\idealin \Cont(\cmp{X})\)
is \(\Cont_0(\N)\idealin \ell^\infty(\N)\).

Composing~\(\varphi\)
with evaluation at \(x\in X\)
gives a state on~\(\Cont_0(Y)\),
that is, a probability measure on~\(Y\).
We may arrange these probability measures to have finite supports that
go to~\(\infty\)
for \(x\to\infty\)
without changing~\(\bdy{\varphi}\).
Hence we may assume this without loss of generality.  Passing to
another subsequence, we may then arrange that the finite supports of
these probability measures for different points in \(X=\N\)
are disjoint.  Then we may replace~\(Y\)
be the disjoint union of these supports.  So if a non-liftable
noncommutative coarse map exists, then it exists in the case where
\(X=\N\)
with the discrete coarse structure, \(Y\)
is a box space \(Y=\bigsqcup Y_n\), and
\[
(\varphi h)(n) = \sum_{y\in Y_n}  c(n,y) h(n,y)
\]
for all \(h\in\Cont_0(Y)\),
where \(c(n,y)\)
for \(y\in Y_n\)
are the point masses of a probability measure on~\(Y_n\)
for each \(n\in\N\).
The challenge is to show that if such a map~\(\varphi\)
induces a \Star{}homomorphism on
\(\Cont(\cmp{Y})/\Cont_0(Y) \to \Cont(\cmp{X})/\Cont_0(X)\),
then there are points \(y_n\in Y_n\)
so that \(\lim_{n\to\infty} (\varphi h)(n) - h(n,y_n)=0\).

\section{Coarse maps from the corona algebra}
\label{sec:coarse_from_boundary}

In this section, we consider noncommutative coarse spaces with nuclear
and separable boundary and \(\sigma\)\nb-compact
interior.  Among such noncommutative coarse spaces, any
\Star{}homomorphism between the boundaries lifts to a noncommutative
coarse map, which is automatically unique up to closeness.  This
follows easily from the following theorem, which compares a given
noncommutative coarse space to the cone over its boundary in
Example~\ref{exa:standard_coarse_space}.

\begin{theorem}
  \label{the:coarse_map_from_standard}
  Let \(A\idealin \cmp{A}\) be a noncommutative coarse space.
  Assume that \(\bdy{A} = B\) is separable and nuclear and
  that~\(A\) is \(\sigma\)\nb-unital.  There is a noncommutative
  coarse equivalence from \(\Cont_0(\N,B) \idealin \Cont(\N^+,B)\)
  to \(A\idealin \cmp{A}\) that induces the identity map on the
  corona algebras.
\end{theorem}

\begin{proof}
  We first construct a noncommutative coarse map from
  \(\Cont_0(\N,B) \idealin \Cont(\N^+,B)\) to \(A\idealin \cmp{A}\)
  that induces the identity map on the corona algebras, following ideas
  of~\cite{Arveson:Extensions_Cstar}.  For this, we need two
  ingredients.  The first is a unital, completely positive section
  \(\sigma\colon \bdy{A}\to\cmp{A}\) for the quotient map
  \(\pi\colon \cmp{A}\prto\bdy{A}\), which exists by the Choi--Effros
  Lifting Theorem because~\(\bdy{A}\) is separable and nuclear
  (see~\cite{Choi-Effros:CP_lifting}).  The second ingredient is a
  sequential approximate unit of~\(A\) that is quasi-central with
  respect to~\(\sigma(\bdy{A})\).  This is an increasing sequence
  \(0= u_0\le u_1 \le u_2 \le \dotsb\) in~\(A\), such that \(\lim
  \norm{u_n a-a}=0\) for all \(a\in A\) and \(\lim \norm{u_n
    \sigma(b) - \sigma(b) u_n}=0\) for all \(b\in B\).  This exists
  because~\(A\) is \(\sigma\)\nb-unital and~\(B\) is separable.  The
  quickest way to deduce this from
  \cite{Arveson:Extensions_Cstar}*{Theorem 1} is by replacing~\(A\)
  by the separable \(\Cst\)\nb-subalgebra~\(A'\) that is generated
  by some sequential approximate unit (which exists by the
  \(\sigma\)\nb-unitality assumption) and~\(\sigma(B)\).

  Let \((F_n)_{n\in\N}\) be an increasing sequence of finite subsets
  of~\(B\) such that \(\bigcup F_n\) is dense in~\(B\); this exists
  because~\(B\) is separable.  We also assume that \(1\in F_0\).  We
  are going to specify a subsequence \(u'_i \defeq u_{n(i)}\) for
  some increasing function \(\N\to\N\), \(i\mapsto n(i)\), with
  \(n(0)=0\).  Let
  \(e_i \defeq (u'_{i+1} - u'_i)^{1/2}\) for \(i\in\N\); this is a
  sequence of positive elements in~\(A\) with \(\sum_{i=0}^{n-1}
  e_i^2 = u'_n\).  Hence \(\sum_{i=0}^\infty e_i^2 = 1\) with strict
  convergence.  By a lemma in~\cite{Arveson:Extensions_Cstar}, for
  each \(\varepsilon>0\) there is \(\delta>0\) such that
  \(\norm{[e_i,\sigma(b)]}<\varepsilon\) holds for all \(b\in B\)
  and \(i\in\N\) with \(\norm{[u'_i,\sigma(b)]} +
  \norm{[u'_{i+1},\sigma(b)]}<\delta\).  Now choose~\(\delta\) for
  \(\varepsilon = 2^{-i}\) and then choose \(n(i)\ge n(i-1)\) so
  that \(\norm{[u'_j,\sigma(b)]} <\delta/2\) for all \(j\ge n(i)\)
  and all \(b\in F_i\).  For this subsequence~\((n(i))_{i\in\N}\),
  we get
  \[
  \norm{[e_i,\sigma(b)]} < 2^{-i}\qquad
  \text{for all }b\in F_i,\ i\in\N.
  \]

  We define
  \[
  \varphi\colon \Cont_0(\N,B) \to A,\qquad
  \varphi(f) \defeq \sum_{i=0}^\infty e_i \sigma(f(i)) e_i.
  \]
  We claim that this sum is norm convergent for \(f\in\Cont_0(\N,B)\).
  It suffices to check this for \(f\ge0\).
  Given \(\varepsilon>0\)
  there is \(n\in\N\)
  with \(\norm{f(i)}\le \varepsilon\) for \(i\ge n\).  Hence
  \[
  0\le \sum_{i=n}^N e_i \sigma(f(i)) e_i
  \le \sum_{i=n}^N e_i \norm{f(i)} e_i
  \le \varepsilon \sum_{i=n}^N e_i^2
  = \varepsilon (u'_{N+1}-u'_n)
  \le \varepsilon
  \]
  if \(N\ge n\).
  This verifies the Cauchy criterion for convergence.  Since
  \(\sigma\)
  is completely positive and \(e_i\ge0\),
  each summand is a completely positive map.  Thus so is~\(\varphi\).

  Let \(1_{[0,n]}\in\Cont_0(\N,B)\) be the characteristic function
  of \([0,n]\); this is an approximate unit in~\(\Cont_0(\N,B)\).
  It is mapped to the sequence
  \(\sum_{i=0}^n e_i^2 = u'_{n+1}\) because~\(\sigma\) is unital.
  Since~\((u'_{n+1})\) is an approximate
  unit in~\(A\), the map~\(\varphi\) is nondegenerate.
  Thus~\(\varphi\) extends uniquely to a strictly continuous map
  from \(\Mult(\Cont_0(\N,B)) \cong \ell^\infty(\N,B)\)
  to~\(\Mult(A)\).  This map still has the form \(f \mapsto
  \sum_{i=0}^\infty e_i \sigma(f(i)) e_i\) for
  \(f\in\ell^\infty(\N,B)\), now with strict convergence of the
  infinite sum.  If \(b\in F_n\), then \(\norm{[e_i,\sigma(b)]} <
  2^{-i}\) for \(i\ge n\).  For the constant function with
  value~\(b\), this implies
  \[
  \sigma(b) - \sum_{i=0}^\infty e_i \sigma(b) e_i
  = \sum_{i=0}^\infty (e_i^2 \sigma(b) - e_i \sigma(b) e_i)
  = \sum_{i=0}^\infty e_i [e_i,\sigma(b)],
  \]
  which is norm-convergent in~\(A\).
  So \(\sigma(b) - \varphi(\const(b)) \in A\)
  for all \(b\in F_n\)
  for all~\(n\),
  and hence for all \(b\in B\);
  here~\(\const b\)
  denotes the constant function \(\N\to B\)
  with value~\(b\).
  Since~\(\varphi\)
  maps \(\Cont_0(\N,B)\)
  and \(\const(b)\)
  for \(b\in B\)
  into~\(\cmp{A}\),
  it maps~\(\Cont(\N^+,B)\)
  to~\(\cmp{A}\).
  The induced map on the corona algebras is the identity map
  \(B\to\bdy{A}\).

  Next we construct a noncommutative coarse~\(\psi\)
  map from \(A\idealin \cmp{A}\)
  to \(\Cont_0(\N,B) \idealin \Cont(\N^+,B)\)
  that induces the identity map \(\bdy{A} \to B\).
  This is inverse to the map constructed above up to closeness, so
  both maps form a coarse equivalence.
  The construction of~\(\psi\) uses the map~\(\varphi\) above, and
  some extra data, namely, an approximation of the identity map
  on~\(B\) by maps of the form \(\beta_k\circ\alpha_k\), \(k\in\N\),
  with unital completely positive maps \(\alpha_k\colon
  B\to\Mat_{n(k)}(\C)\) and \(\beta_k\colon \Mat_{n(k)}(\C)\to B\)
  for some matrix sizes~\(n(k)\) for \(k\in\N\); this exists
  because~\(B\) is nuclear and separable.  Our Ansatz is to define
  \[
  \psi\colon A\to \Cont_0(\N,B),\qquad
  \psi(a)(k) \defeq \beta_k\gamma_k(a)
  \quad\text{for }k\in\N,\ a\in A,
  \]
  where the maps \(\gamma_k\colon A\to\Mat_{n(k)}(\C)\) are
  completely positive contractions with the following extra
  properties:
  \begin{enumerate}
  \item \label{enum:gammak_1} \(\lim_{k\to\infty} \gamma_k(u'_n)=0\)
    for each \(n\in\N\);
  \item \label{enum:gammak_2} \(\lim_{n\to\infty} \gamma_k(u'_n)=1\)
    for each \(k\in\N\);
  \item \label{enum:gammak_3} let \(\Mult(\gamma_k)\colon
    \Mult(A)\to\Mat_{n(k)}(\C)\) be the unique strictly continuous
    extension of~\(\gamma_k\) (see below for its existence); then
    \[
    \lim_{k\to\infty}
    \abs{\Mult(\gamma_k)(\varphi(\const b)) - \alpha_k(b)}
    =0
    \]
    for all \(b\in B\), where \(\varphi\colon \Cont(\N^+,B)\to
    \cmp{A}\subseteq \Mult(A)\) is the map built above.
  \end{enumerate}
  First we show that~\(\psi\) is a noncommutative coarse map
  that induces the identity map on the corona algebra if the
  maps~\((\gamma_k)_{k\in \N}\) have the properties listed above.
  Since the maps \(\gamma_k\) and~\(\beta_k\) are completely
  positive contractions, so are the maps \(\beta_k \gamma_k\).
  Hence~\(\psi\) is a well-defined completely positive contraction
  \(A\to\ell^\infty(\N,B)\).  It maps the approximate
  unit~\((u_n')\) into~\(\Cont_0(\N,A)\) by
  property~\ref{enum:gammak_1}.  Then
  \begin{multline*}
    \psi(a (u_n')^{1/2})^* \psi(a (u_n')^{1/2})
    \le \psi\bigl((u_n')^{1/2} a^* a (u_n')^{1/2}\bigr)
    \\\le \norm{a} \psi\bigl((u_n')^{1/2} (u_n')^{1/2}\bigr)
    \in \Cont_0(\N,B)
  \end{multline*}
  for all \(a\in A\) by \cite{Lance:Hilbert_modules}*{Lemma 5.3}.
  Since \(\Cont_0(\N,B)\) is an ideal in \(\ell^\infty(\N,B)\), this
  implies \(\psi(a (u_n')^{1/2})\in \Cont_0(\N,B)\) for \(a\in A\),
  and then \(\psi(a)\in \Cont_0(\N,B)\) because \((u_n')^{1/2}\) is
  an approximate unit.  Thus \(\psi(A)\subseteq \Cont_0(\N,B)\).

  Since~\(\psi\) is completely positive, the sequence \(\psi(u_n')\)
  in \(\Cont_0(\N,B)\) is increasing.  Property~\ref{enum:gammak_2}
  of the maps~\(\gamma_k\) says that \(\psi(u_n')(k) \to 1\) for
  each \(k\in\N\) because \(\beta_k(1)=1\).  Thus \(\psi(u_n')\) is
  an approximate unit in~\(\Cont_0(\N,B)\), so~\(\psi\) is
  nondegenerate.  Thus~\(\psi\) has a unique strictly continuous
  extension \(\Mult(\psi)\colon \Mult(A) \to \Mult(\Cont_0(\N,B))
  \cong \ell^\infty(\N,B)\).  Evaluation at \(k\in\N\) gives
  \(\Mult(\psi)(a)(k) = \beta_k\circ \Mult(\gamma_k)(a)\) for the
  unique strictly continuous extension~\(\Mult(\gamma_k)\)
  of~\(\gamma_k\).

  We have \(\cmp{A} = A + \varphi(\const B)\)
  because \(\cmp{A}/A = \bdy{A} = B\)
  and \(B\ni b\mapsto \varphi(\const b)\in \cmp{A}\)
  is a section for this extension.  Since
  \(\psi(A)\subseteq \Cont_0(\N,B)\),
  we have \(\Mult(\psi)(\cmp{A})\subseteq \Cont(\N^+,B)\)
  if and only if \(\Mult(\psi)(\varphi(\const b))\in \Cont(\N^+,B)\)
  for all \(b\in B\).  Property~\ref{enum:gammak_3} implies
  \[
  \lim_{k\to\infty} \Mult(\psi)(\varphi(\const(b)))
  = \lim_{k\to\infty} \beta_k \alpha_k(b)
  = b
  \]
  because \(\norm{\beta_k}\le 1\) for all \(k\in\N\).  Thus
  \(\psi(\cmp{A})\subseteq \Cont(\N^+,B)\) and~\(\psi\) induces the
  identity map \(\bdy{A}\to B\) on the corona algebras.
  Hence~\(\psi\) and~\(\varphi\) give a coarse equivalence between
  \(A\idealin \cmp{A}\) and \(\Cont_0(\N,B) \idealin
  \Cont(\N^+,B)\).

  It remains to find maps \(\gamma_k\colon A\to\Mat_{n(k)}(\C)\)
  that satisfy \ref{enum:gammak_1}--\ref{enum:gammak_3}.  First we fix
  \(k\in\N\).
  A completely positive linear map
  \(\gamma_k\colon A\to\Mat_{n(k)}(\C)\)
  is equivalent to a positive linear functional
  \(\tilde\gamma_k\colon \Mat_{n(k)}(A)\to \C\),
  see \cite{Paulsen:Completely_bounded}*{Chapter 6}.  Moreover,
  \(\norm{\gamma_k} = \norm{\tilde\gamma_k}\).
  The same construction for~\(B\)
  turns the given completely positive unital maps~\(\alpha_k\)
  into states \(\tilde\alpha_k\colon \Mat_{n(k)}(B)\to\C\).
  Any positive linear functional on a \(\Cst\)\nb-algebra~\(A\)
  extends to the multiplier algebra~\(\Mult(A)\)
  by extending its GNS-representation, and this extension is strictly
  continuous.  Then the corresponding completely positive linear map
  \(A\to\Mat_{n(k)}(\C)\)
  is strictly continuous as well.  Thus any completely positive map
  \(\gamma_k\colon A\to\Mat_{n(k)}(\C)\)
  is strictly continuous and extends to~\(\Mult(A)\).

  The unital, completely positive map \(b\mapsto \varphi(\const b)\)
  from~\(B\) to~\(\cmp{A}\) is faithful because it induces an
  isomorphism \(B\to\bdy{A}\).  Therefore, a self-adjoint element
  \(b\in \Mat_{n(k)}(B)\) is positive if and only if
  \(\varphi^{(n(k))}(\const b) \ge0\).  Equivalently,
  \(l\bigl(\varphi^{(n(k))}(\const b)\bigr) \ge0\) for every
  state~\(l\) on~\(\Mat_{n(k)}(A)\), where we also write~\(l\) for
  the unique strictly continuous extension of~\(l\)
  to~\(\Mat_{n(k)}(\Mult(A))\).  Therefore, any state
  on~\(\Mat_{n(k)}(B)\) is contained in the weak\(^*\)-closed convex
  hull of the set of states of the form~\(\varphi^*(l)\) with
  \[
  \varphi^*(l)(b)\defeq l\bigl(\varphi^{(n(k))}(\const b)\bigr)
  \]
  for states~\(l\) on~\(\Mat_{n(k)}(A)\) (see
  \cite{Dixmier:Cstar-algebras}*{Lemma 3.4.1}).  Since this set of
  states is already convex, any state on~\(\Mat_{n(k)}(B)\) is a
  weak\(^*\)-limit of states of the form~\(\varphi^*(l)\) for
  states~\(l\) on~\(\Mat_{n(k)}(A)\).  In particular, we may
  approximate the state \(\tilde\beta_k\colon \Mat_{n(k)}(B)\to\C\)
  pointwise by states of the form \(\varphi^*(\tilde\gamma_k)\) for
  a state \(\tilde\gamma_k\colon \Mat_{n(k)}(A)\to\C\).

  Thus we may approximate \(\beta_k\colon B\to\Mat_{n(k)}(\C)\)
  pointwise by maps of the form
  \(\Mult(\gamma_k)\circ\varphi\circ\const\) with completely
  positive contractions \(\gamma_k\colon A\to\Mat_{n(k)}(\C)\).
  Choose the increasing sequence of finite subsets \(F_i\subseteq
  B\) as above.  There is a completely positive contraction
  \(\gamma_k\colon A\to \Mat_{n(k)}(\C)\) with
  \(\abs{\Mult(\gamma_k)\circ \varphi(\const b) - \beta_k(b)} <
  2^{-k}\) for \(b\in F_k\).  Choosing such~\(\gamma_k\) for each
  \(k\in\N\), we get maps~\(\gamma_k\) that
  verify~\ref{enum:gammak_3} for all \(b\in \bigcup F_i\) and hence
  for all \(b\in B\).  Since the map \(\tilde\gamma_k\colon
  \Mat_{n(k)}(A)\to\C\) corresponding to~\(\gamma_k\) is a state,
  \(\lim_{i\to\infty} \tilde\gamma_k(u'_i\cdot 1_{\Mat_{n(k)}\C}) =
  1\).  This is equivalent to \(\lim_{i\to\infty} \gamma_k(u'_i)=
  1\).  Thus~\ref{enum:gammak_2} is also built into our
  construction.  To also verify~\ref{enum:gammak_1}, we refine the
  construction above slightly.
  The map \(\varphi\circ\const\colon B\to \cmp{A}\) remains faithful
  when we project to~\(\Mult(A)/A\).  Hence for each \(i\ge0\),
  \(\varepsilon>0\), states~\(l\) of \(\Mat_{n(k)}(A)\) with
  \(l(u'_i)<\varepsilon\) still detect whether self-adjoint elements
  of~\(\Mat_{n(k)}(B)\) are positive.  So we may choose~\(\gamma_k\)
  above so that, say, \(\gamma_k(u'_k)<2^{-k}\).  Then also
  \(\gamma_k(u'_i)<2^{-k}\) for \(i\le k\) because \(u'_i\le u'_k\).
  Thus we have also arranged for~\ref{enum:gammak_1} to hold.
\end{proof}

In particular, if~\(\bdy{A}\) is separable and commutative, then
Theorem~\ref{the:coarse_map_from_standard} gives a coarse
equivalence between \(A\idealin \cmp{A}\) and a commutative coarse
space.  Our proof needs separability of~\(\bdy{A}\) in a crucial
way.  It seems unlikely that similar results hold for non-separable
commutative boundaries.

\begin{theorem}
  \label{the:lift_boundary_map}
  Let \(A_1\idealin \cmp{A_1}\) and \(A_2\idealin \cmp{A_2}\) be
  noncommutative coarse spaces.  Assume that the corona
  algebra~\(\bdy{A_1}\) is nuclear and separable and that \(A_1\)
  and~\(A_2\) are \(\sigma\)\nb-unital.  Then any
  \Star{}homomorphism \(f\colon \bdy{A_1}\to\bdy{A_2}\) lifts to a
  noncommutative coarse map from \(A_1\idealin \cmp{A_1}\) to
  \(A_2\idealin \cmp{A_2}\).  All such liftings are close.
\end{theorem}

\begin{proof}
  Theorem~\ref{the:coarse_map_from_standard} gives a coarse
  equivalence between \(A_1\idealin \cmp{A_1}\)
  and \(\Cont_0(\N,\bdy{A_1})\idealin \Cont(\N^+,\bdy{A_1})\).
  Since~\(\bdy{A_1}\)
  is separable and nuclear, the Choi--Effros Lifting Theorem lifts the
  \Star{}homomorphism \(f\colon \bdy{A_1}\to\bdy{A_2}\)
  to a completely positive, unital map
  \(\sigma_f\colon \bdy{A_1}\to \cmp{A_2}\).
  Let \((u_n)_{n\in\N}\)
  be a sequential approximate unit for~\(A_2\).
  Let \(A_2' \subseteq A_2\)
  be the \(\Cst\)\nb-subalgebra
  generated by elements of the form \(u_n\cdot \sigma_f(x)\)
  and \(\sigma_f(x)\cdot u_n\)
  for \(n\in\N\),
  \(x\in\bdy{A_1}\).
  Since~\(\bdy{A_1}\)
  is separable, so is~\(A_2'\).
  The approximate unit~\((u_n)\)
  lies in~\(A_2'\)
  and is an approximate unit there, so the map
  \(A_2'\hookrightarrow A_2\)
  is nondegenerate and induces an inclusion
  \(\Mult(A_2') \hookrightarrow \Mult(A_2)\).
  The elements \(\sigma_f(x)\in\cmp{A_2}\)
  for \(x\in \bdy{A_1}\)
  are also multipliers of~\(A_2'\).
  So \(\cmp{A_2'} = A_2' + \sigma_f(B)\subseteq \cmp{A_2}\)
  gives another noncommutative coarse space
  \(A_2'\idealin \cmp{A_2'}\)
  contained in \(A_2\idealin \cmp{A_2}\).
  Since \(\cmp{A_2'}/A_2'\) as a quotient of~\(\bdy{A_1}\)
  is separable and nuclear, this new noncommutative coarse space
  \(A_2'\idealin \cmp{A_2'}\)
  is also coarsely equivalent to
  \(\Cont_0(\N,\bdy{A_1})\idealin \Cont(\N^+,\bdy{A_1})\).
  Composing the coarse equivalences from \(A_1\idealin \cmp{A_1}\)
  to \(\Cont_0(\N,\bdy{A_1})\idealin \Cont(\N^+,\bdy{A_1})\)
  and on to \(A_2'\idealin \cmp{A_2'}\)
  with the inclusion \(A_2'\hookrightarrow A_2\)
  gives the desired lifting of~\(f\).
  Two noncommutative coarse maps are close if and only if they induce
  the same map on the boundaries.
\end{proof}

\begin{remark}
  \label{rem:coarse_map_bdy}
  The first part of the proof of
  Theorem~\ref{the:coarse_map_from_standard} still gives a
  noncommutative coarse map from \(\Cont_0(\N,B) \idealin
  \Cont(\N^+,B)\) to \(A\idealin \cmp{A}\) that induces the identity
  map on the corona algebras, assuming only that there is a
  completely positive unital section \(\bdy{A}\to\cmp{A}\) and a
  sequential approximate unit for~\(A\) that is quasi-central
  in~\(\cmp{A}\).  These sufficient assumptions are also necessary.
  First, since the boundary quotient map \(\Cont(\N^+,B) \prto B\)
  has an obvious completely positive, unital section for any~\(B\),
  a noncommutative coarse map from \(\Cont_0(\N,B) \idealin
  \Cont(\N^+,B)\) to \(A\idealin \cmp{A}\) that induces the identity
  map on the boundary can only exist if \(\cmp{A}\prto \bdy{A}\) has
  a completely positive, unital section.  Secondly,
  \((1_{[0,n]})_{n\in\N}\) is a quasi-central approximate unit for
  \(\Cont_0(\N,B) \idealin \Cont(\N^+,B)\) because~\(B\) is unital.
  A noncommutative coarse map to \(A\idealin \cmp{A}\) maps it to an
  approximate unit for~\(A\) because it is nondegenerate.  A
  computation using the Stinespring Dilation and that the induced
  map on the corona algebras is an isomorphism shows that the
  image of this approximate unit is quasi-central with respect
  to~\(\cmp{A}\).

  If~\(\cmp{A}\) is commutative and~\(A\) \(\sigma\)\nb-unital, then
  there certainly exists a quasi-central approximate unit.  We do
  not know, however, whether there is a completely positive section
  \(\bdy{A} \to \cmp{A}\): the Choi--Effros Lifting Theorem only
  applies if~\(\bdy{A}\) is separable.
\end{remark}

\section{Lifting maps between metrisable boundaries}
\label{sec:lift_metrisable}

The following theorem is a version of
Theorem~\ref{the:lift_boundary_map} for ordinary coarse
spaces.  We have not seen this in the literature.  Recall the
category of compactified spaces~\(\Compactified_p\) introduced in
Section~\ref{sec:coarse_maps_commutative}.

\begin{theorem}
  \label{the:lift_metrisable}
  Let \(X \subseteq \cmp{X}\) and \(Y\subseteq \cmp{Y}\) be objects
  of~\(\Compactified_p\).  Let \(\varphi\colon
  \bdy{X}\to\bdy{Y}\) be a continuous map.  If the
  boundary~\(\bdy{Y}\) is metrisable, then there is a morphism
  \(F\colon X\to Y\) in~\(\Compactified_p\) with boundary
  values~\(\varphi\), and any two such morphisms are close.  The
  map~\(F\) is coarse with respect to the coarse structures defined
  by the compactifications.
\end{theorem}

\begin{proof}
  It suffices to produce morphisms in~\(\Compactified_p\).  These
  are coarse maps by Lemma~\ref{lem:functoriality}.  We need to find
  a map \(F\colon X\to Y\) such that the map \(\cmp{F}\colon
  \cmp{X}\to\cmp{Y}\) given by \(F\) on~\(X\) and~\(\varphi\)
  on~\(\bdy{X}\) is continuous on~\(\bdy{X}\).  By definition, two
  morphisms in~\(\Compactified_p\) are close if and only if they
  have the same boundary values.  Thus any two such morphisms are
  close.  The construction of~\(F\) needs some preparations.  At
  first, we assume \(\cmp{X}\) and~\(\cmp{Y}\) to be second
  countable.  We will remove these extra assumptions later.

  By assumption, the topologies on \(\cmp{X}\)
  and~\(\cmp{Y}\)
  are metrisable, that is, they may be defined by metrics
  \(d_{\cmp{X}}\)
  and~\(d_{\cmp{Y}}\) on \(\cmp{X}\) and~\(\cmp{Y}\), respectively.
  Since \(X\) and~\(Y\) are \(\sigma\)\nb-compact, there are
  increasing sequences of compact subsets \((K_n)_{n\in\N}\)
  and~\((L_n)_{n\in\N}\) with \(K_0=\emptyset\), \(X\defeq \bigcup_n
  K_n\), \(L_0=\emptyset\), \(Y\defeq \bigcup_n L_n\).  We fix
  these.
  The balls
  \[
  B^{\bdy{Y}}(y, 2^{-n})\defeq
  \{y' \in \bdy{Y} : d_{\cmp{Y}}(y', y) < 2^{-n}\}
  \]
  for \(y\in \bdy{Y}\) and fixed~\(n\) form an open cover
  of~\(\bdy{Y}\).  Since~\(\bdy{Y}\) is compact, there are finitely
  many points \(y_{n,1}, y_{n,2}, \dotsc, y_{n,l}\) such that
  \(\bdy{Y} = \bigcup_{i=1}^l B^{\bdy{Y}}(y_{n,i}, 2^{-n})\).  Since
  \(Y\subseteq \cmp{Y}\) is dense, \(Y\setminus L_n\) is dense in
  \(\cmp{Y}\setminus L_n\).  Since \(B^{\bdy{Y}}(y_{n,i},2^{-n})\)
  for \(i\in\{1,\dotsc,l\}\) is an open neighbourhood of
  \(y_{n,i}\in \cmp{Y}\setminus L_n\), there is \(y'_{i,n} \in
  (Y\setminus L_n)\cap B^{\bdy{Y}}(y_{n,i},2^{-n})\).
  Now we can define the map \(F\colon X\to Y\).  Let \(x\in X\).
  Since \(K_0=\emptyset\) and \(X\defeq \bigcup_n K_n\), there is a
  unique \(n\in\N\) with \(x\in K_{n+1}\setminus K_n\).  We first
  choose some point~\(\delta(x)\) in~\(\bdy{X}\) that is closest
  to~\(x\) with respect to the metric~\(d_{\cmp{X}}\).  Since
  \(\bdy{Y} = \bigcup_{i=1}^l B^{\bdy{Y}}(y_{n,i}, 2^{-n})\), there
  is \(i\in\{1,\dotsc,l\}\) with
  \(d_{\cmp{Y}}(\varphi(\delta(x)),y_{n,i}) <2^{-n}\).  We pick such
  an~\(i\) and let \(F(x)\defeq y'_{n,i} \in Y\).  This construction
  is illustrated in Figure~\ref{fig:boundary_extend}.
  \begin{figure}[htbp]
    \centering
    \begin{tikzpicture}
      \begin{scope}[shift={(0,.5)},scale=.5]
        \draw[very thick, name path=A] (0,5) to [out = 90, in = 195] (5, 8) to [out = 15, in = 90] (8, 2) to [out = 270, in = 15] (2,0) to [out = 195, in = 270] (0,5);
        \draw (0.5,2) to [out = 90, in = 195] (6, 7) to [out = 15, in = 90] (7, 2) to [out = 270, in = 15] (2,0.5) to [out = 195, in = 270] (0.5,2);
        \draw (2,2) to [out = 90, in = 195] (5, 5) to [out = 15, in = 90] (6, 2) to [out = 270, in = 15] (2,1) to [out = 195, in = 270] (2,2);
        \draw (3,3) to [out = 90, in = 195] (4, 4) to [out = 15, in = 90] (5, 2) to [out = 270, in = 0] (4,2.5) to [out = 180, in = 270] (3,3);
        \node [right] at (4, 5.5) {$\scriptstyle{K_{n+1}\setminus K_n}$};
        \node [left] at (5, 8.5) {$\bdy{X}$};

        \coordinate (pt) at (2.5, 5);
        \node [right] at (2.5, 5){$x$};
        \filldraw[black] (pt) circle[radius = 2pt];

        \coordinate (nearpt) at (0, 5);
        \node [left] at (0, 5)  {$x'$};
        \filldraw[black] (nearpt) circle[radius = 2pt];
      \end{scope}

      \draw[->] (4.1,3) -- (4.7, 3);

      \node [below] at (2.5, 0) {$\cmp{X}$};

      \node [below] at (7.5, 0) {$\cmp{Y}$};

      \coordinate (center_Y) at (7.5, 3);
      \begin{scope}[shift=(center_Y),scale=.7]
        \draw [very thick, name path= cir] (0,0) circle (3cm);
        \draw [fill= gray!25] (0,0) circle (2cm);

        \node[above] at (0, 0) {$L_n$};
        \node[above] at (70:3cm) {$\bdy{Y}$};

        \coordinate (pt1) at (140:3cm);
        \filldraw[black] (pt1) circle[radius = 2pt];
        \node [left] at (pt1) {$y_{n, 1}$};

        \coordinate (pt2) at (100:3cm);
        \filldraw[black] (pt2) circle[radius = 2pt];
        \node [above] at (pt2) {$y_{n, 2}$};

        \coordinate (pt3) at (50:3cm);
        \filldraw[black] (pt3) circle[radius = 2pt];
        \node [right] at (pt3)  {$y_{n, 3}$};

        \coordinate (pt4) at (360:3cm);
        \filldraw[black] (pt4) circle[radius = 2pt];
        \node [right] at (pt4) {$y_{n, 4}$};

        \coordinate (pt5) at (320:3cm);
        \filldraw[black] (pt5) circle[radius = 2pt];
        \node [right] at (pt5) {$y_{n, 5}$};

        \coordinate (pt6) at (285:3cm);
        \filldraw[black] (pt6) circle[radius = 2pt];
        \node [below] at (pt6) {$y_{n, 6}$};

        \coordinate (pt7) at (240:3cm);
        \filldraw[black] (pt7) circle[radius = 2pt];
        \node [left] at (pt7) {$y_{n, 7}$};

        \coordinate (pt8) at (190:3cm);
        \filldraw[black] (pt8) circle[radius = 2pt];
        \node [left] at (pt8) {$y_{n, 8}$};

        \coordinate (img) at (250:3cm);
        \filldraw[black] (img) circle[radius = 2pt];
        \node [below] at (img) {$\varphi(x')$};

        \coordinate (img) at (228:2.2cm);
        \filldraw[black] (img) circle[radius = 2pt];

        \foreach \i in {1,2,6,7,8}
        {
          \path[name path=nbd\i] (pt\i) circle (1.5cm);
          \path[name intersections={of={cir} and {nbd\i},by={u\i,l\i}}];
          \pgfpathmoveto{\pgfpointanchor{l\i}{north}}
          \pgfpatharcto{1.5cm}{1.5cm}{0}{0}{1}{\pgfpointanchor{u\i}{south}}
          \pgfusepath{stroke}
        }
        \foreach \i in {3,4,5}
        {
          \path[name path=nbd\i] (pt\i) circle (1.5cm);
          \path[name intersections={of={cir} and {nbd\i},by={u\i,l\i}}];
          \pgfpathmoveto{\pgfpointanchor{u\i}{north}}
          \pgfpatharcto{1.5cm}{1.5cm}{0}{0}{1}{\pgfpointanchor{l\i}{south}}
          \pgfusepath{stroke}
        }
        \node [above right=2pt,narrowfill] at (img) {$\scriptstyle y'_{n,7}= F(x)$};
      \end{scope}
    \end{tikzpicture}
    \caption{Extension of a continuous map~\(\varphi\)
      from the boundary to a map~\(F\) on the interior.}
    \label{fig:boundary_extend}
  \end{figure}
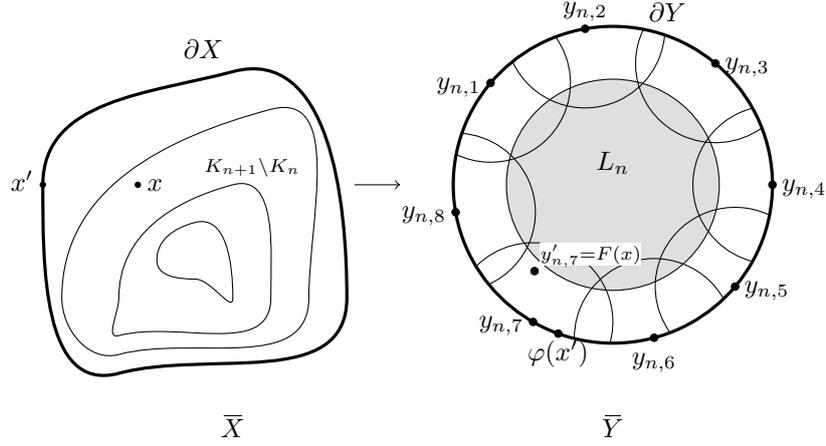

  We claim that the map \(\cmp{F}\colon \cmp{X}\to\cmp{Y}\) defined
  by \(\cmp{F}(x)= F(x)\) if \(x\in X\) and \(\cmp{F}(x)=
  \varphi(x)\) if \(x\in \bdy{X}\) is continuous on~\(\bdy{X}\).
  Let \(x_\infty\in\bdy{X}\).  We prove continuity at~\(x_\infty\).
  Since~\(\cmp{X}\) is metrisable, it suffices to check \(\lim
  \cmp{F}(x_n) = \cmp{F}(x_\infty) = \varphi(x_\infty)\) for any
  sequence~\((x_n)_{n\in\N}\) in~\(\cmp{X}\) converging
  to~\(x_\infty\).  Since~\(\varphi\) is
  continuous, it suffices to consider sequences in~\(X\).  Since
  \(\lim d_{\cmp{X}}(x_n,x_\infty)=0\), the points
  \(\delta(x_n)\in\bdy{X}\) closest to~\(x_n\) chosen in the
  construction of~\(F\) satisfy \(\lim
  d_{\cmp{X}}(\delta(x_n),x_\infty)=0\).  Hence \(\lim
  d_{\cmp{Y}}(\varphi(\delta(x_n)), \varphi(x_\infty))= 0\).  If
  \(x_n\in K_{m+1} \setminus K_m\), then \(d_{\cmp{Y}}\bigl(F(x_n),
  \varphi(\delta(x_n))\bigr) \le 2^{-m}\).  Since \((x_n)_{n\in\N}\)
  converges to a boundary point, \(m\to\infty\) as \(n\to\infty\).
  Thus both \(d_{\cmp{Y}}\bigl(F(x_n), \varphi(\delta(x_n))\bigr)\)
  and \(d_{\cmp{Y}}\bigl(\varphi(\delta(x_n)),
  \varphi(x_\infty)\bigr)\) converge to~\(0\) as \(n\to\infty\).
  The triangle inequality shows that \(\lim F(x_n)=
  \varphi(x_\infty)\) as desired.

  Now we generalise the result to the case where \(X\)
  and~\(Y\)
  are \(\sigma\)\nb-compact
  and~\(\bdy{Y}\)
  is second countable.  We have relegated the more technical parts of
  the proof to the appendix.  The image of~\(\Cont(\bdy{Y})\)
  in~\(\Cont(\bdy{X})\)
  is a separable \(\Cst\)\nb-subalgebra,
  which corresponds to a second countable quotient of~\(\bdy{X}\)
  as in Lemma~\ref{lem:convergence}.
  Lemma~\ref{lem:convergence}.\ref{en:convergence1} gives a
  compactification \(X'\subseteq \cmp{X'}\)
  and a continuous quotient map \(\varrho\colon \cmp{X}\to\cmp{X'}\)
  with \(\varrho(X)\subseteq X'\)
  and \(\varrho(\bdy{X})\subseteq \bdy{X'}\)
  such that~\(\cmp{X'}\)
  is second countable and such that
  \(\varrho^*\Cont(\bdy{X'}) \subseteq \Cont(\bdy{X})\)
  is the image \(\varphi^*(\Cont(\bdy{Y}))\)
  of~\(\Cont(\bdy{Y})\).
  Thus \(\varphi\colon \bdy{X} \to \bdy{Y}\)
  factors through a continuous map
  \(\varphi'\colon \bdy{X'} \to \bdy{Y}\).
  If we can extend~\(\varphi'\)
  to a map \(\cmp{F'}\colon \cmp{X'} \to \cmp{Y}\)
  that is continuous on~\(\bdy{X'}\),
  then \(\cmp{F}\defeq \cmp{F'}\circ \varrho\colon \cmp{X}\to\cmp{Y}\)
  is the desired extension of~\(\varphi\)
  that is continuous on~\(\bdy{X}\).
  Thus it suffices to prove the theorem for \(X'\subseteq \cmp{X'}\)
  instead of \(X\subseteq \cmp{X}\).
  Thus it is no loss of generality to assume~\(\cmp{X}\)
  to be second countable.

  Similarly, since~\(Y\)
  is \(\sigma\)\nb-compact
  and~\(\bdy{Y}\)
  is second countable,
  Lemma~\ref{lem:convergence}.\ref{en:convergence1} gives a
  compactification \(Y'\subseteq \cmp{Y'}\)
  and a continuous quotient map \(\varrho\colon \cmp{Y}\to\cmp{Y'}\)
  with \(\varrho(Y)\subseteq Y'\)
  such that~\(\cmp{Y'}\)
  is second countable and~\(\varrho\)
  restricts to a homeomorphism \(\bdy{Y}\congto \bdy{Y'}\).
  The proof above applies to \(Y'\subseteq \cmp{Y'}\),
  which is second countable, so we get a map
  \(\cmp{F'}\colon \cmp{X}\to\cmp{Y'}\)
  with \(\cmp{F}(X)\subseteq Y'\)
  that extends \(\varrho\circ\varphi\)
  on~\(\bdy{X}\)
  and that is continuous on~\(\bdy{X}\).
  Since \(Y\to Y'\)
  is surjective, we may lift~\(\cmp{F'}\)
  to a map \(\cmp{F}\colon \cmp{X}\to\cmp{Y}\).
  This extends~\(\varphi\)
  on~\(\bdy{X}\)
  because \(\varrho|_{\bdy{Y}}\)
  is bijective.  We claim that it is continuous on~\(\bdy{X}\).
  It suffices to check that \(\lim F(x_\alpha) = \varphi(x_\infty)\)
  for any net~\((x_\alpha)\)
  in~\(X\)
  with \(\lim x_\alpha = x_\infty\in \bdy{X}\).
  Since~\(\cmp{F'}\)
  is continuous,
  \(\lim \varrho\circ F(x_\alpha) = \varrho\circ \varphi(x_\infty)\).
  This implies \(\lim F(x_\alpha) = \varphi(x_\infty)\)
  by Lemma~\ref{lem:convergence}.\ref{en:convergence2}.
\end{proof}

\appendix

\section{Compactifications with non-metrisable interiors}

The lemmas below help to extend
Proposition~\ref{pro:compactification2coarse_metrisable} and
Theorem~\ref{the:lift_metrisable} from second countable
compactifications to compactifications with $\sigma$\nb-compact
interior and second countable boundary.

\begin{lemma}
  \label{lem:convergence}
  Let~$X$ be a $\sigma$\nb-compact, locally compact space with a
  compactification~$\cmp{X}$.  Let~\(\bdy{Y}\)
  be a second countable quotient of the boundary~$\bdy{X}$.
  \begin{enumerate}
  \item \label{en:convergence1} There are a compactification
    \(X'\subseteq \cmp{X'}\)
    and a continuous map \(\varrho\colon \cmp{X}\to\cmp{X'}\)
    with \(\varrho(X)\subseteq X'\)
    and \(\varrho(\bdy{X})\subseteq \bdy{X'}\),
    such that~\(\cmp{X'}\)
    is second countable and the induced map
    \(\bdy{\varrho}\colon \bdy{X}\to \bdy{X'}\)
    is the given map \(\bdy{X} \to \bdy{Y}\).
  \item \label{en:convergence2} Assume \(\bdy{X}=\bdy{Y}\)
    and let \(X'\subseteq \cmp{X'}\)
    and~\(\varrho\)
    be as in~\ref{en:convergence1}.  Let \(x_\infty\in \bdy{X}\)
    and let~\((x_\alpha)_{\alpha\in I}\)
    be a net in~\(X\).
    If \(\lim \varrho(x_\alpha) = \varrho(x_\infty)\),
    then \(\lim x_\alpha = x_\infty\).
  \item \label{en:convergence3} If \(\bdy{X}=\bdy{Y}\)
    is second countable, then any \(x_\infty\in \bdy{X}\)
    is the limit of a convergent sequence~\((x_n)\) in~\(X\).
  \end{enumerate}
\end{lemma}

\begin{proof}
  We prove~\ref{en:convergence1}.  Since~\(\bdy{Y}\)
  is second countable, \(\Cont(\bdy{Y})\)
  contains a dense sequence~\((f_n)_{n\in\N}\).
  The restriction map \(\Cont(\cmp{X})\to\Cont(\bdy{X})\)
  is surjective by the Tietze Extension Theorem.  Therefore, there are
  functions \(g_n\in \Cont(\cmp{X})\)
  for \(n\in\N\)
  with \(g_n|_{\bdy{X}} = f_n\).
  Since~\(X\)
  is \(\sigma\)\nb-compact,
  there is a countable approximate unit~\((u^X_n)_{n\in\N}\)
  in~\(\Cont_0(X)\).
  Let \(\cmp{A}\subseteq \Cont(\cmp{X})\)
  be the \(\Cst\)\nb-subalgebra
  generated by \(\{1,g_n,u^X_n\mid n\in\N\}\).
  Let \(\bdy{A} \subseteq \Cont(\bdy{X})\)
  be the image of~\(\cmp{A}\)
  under the quotient map \(\Cont(\cmp{X})\to\Cont(\bdy{X})\),
  and let \(A= \ker(\bdy{A} \prto \bdy{A})\).
  By construction, \(\cmp{A}\)
  and~\(\bdy{A}\) are unital, separable, and commutative.  So
  \[
  A \cong \Cont(X'),\qquad
  \cmp{A} = \Cont(\cmp{X'}),\qquad
  \bdy{A} = \Cont(\bdy{X'})
  \]
  for a second countable, compact space~\(\cmp{X'}\), an open
  subspace \(X' \subseteq \cmp{X'}\), and \(\bdy{X'} \defeq
  \cmp{X'}\setminus X'\).  By construction, we have a morphism of
  extensions
  \[
  \begin{tikzpicture}[baseline=(current bounding box.west)]
    \matrix (m) [cd,column sep=1em] {
      \Cont_0(X) & {\Cont(\cmp{X})} & \Cont(\bdy{X})\\
      \Cont_0(X') & {\Cont(\cmp{X'})} & \Cont(\bdy{X'}).\\
    };
    \draw[>->] (m-1-1) -- (m-1-2);
    \draw[>->] (m-2-1) -- (m-2-2);
    \draw[->>] (m-1-2) -- (m-1-3);
    \draw[->>] (m-2-2) -- (m-2-3);
    \draw[right hook->] (m-2-1) -- (m-1-1);
    \draw[right hook->] (m-2-2) -- (m-1-2);
    \draw[right hook->] (m-2-3) -- (m-1-3);
  \end{tikzpicture}
  \]
  The inclusion map \(\Cont_0(X')\hookrightarrow \Cont_0(X)\)
  is nondegenerate because its image contains an approximate unit
  for~\(\Cont_0(X)\).
  Since the extension in the top row is essential, so is that in the
  bottom row.  That is, \(X'\subseteq \cmp{X'}\)
  is dense.  So~\(\cmp{X'}\)
  is a compactification of~\(X'\)
  with boundary~\(\bdy{X'}\).
  The inclusion map \(\Cont(\cmp{X'}) \hookrightarrow \Cont(\cmp{X})\)
  corresponds to a quotient map
  \(\varrho\colon \cmp{X}\prto\cmp{X'}\),
  which maps~\(X\)
  to~\(X'\)
  and \(\bdy{X}\)
  to~\(\bdy{X'}\).
  By construction,
  \(\Cont(\bdy{Y}) = \Cont(\bdy{X'}) \subseteq \Cont(\bdy{X})\).
  So we may identify \(\bdy{X'} = \bdy{Y}\).

  We prove~\ref{en:convergence2}.  Any quotient as
  in~\ref{en:convergence1} comes from a compactification
  \(A \into \cmp{A} \prto \bdy{A}\)
  as in the proof of~\ref{en:convergence1}, so we may assume this.
  Since~\(A\)
  is separable, \(X'\)
  is metrisable, so we may define its topology by a metric~\(d'\).
  Then \(\varrho^*(d')\)
  is a continuous quasi-metric on~\(X\).
  Since \(\lim \varrho(x_\alpha) = \varrho(x_\infty)\in\bdy{X'}\),
  we have \(\lim f(\varrho(x_\alpha)) = 0\)
  for all \(f\in A\).
  This implies \(\lim f(x_\alpha) = 0\)
  for all \(f\in \Cont_0(X)\)
  because \(\varrho^*\colon A\to \Cont_0(X)\)
  is nondegenerate.  This rules out an accumulation point
  of~\((x_\alpha)\)
  in~\(X\).
  Let \(y\in\bdy{X}\setminus\{x_\infty\}\).
  Then \(\varrho(y)\neq \varrho(x_\infty)\)
  because~\(\varrho|_{\bdy{X}}\)
  is a homeomorphism.  By Urysohn's Lemma, there is
  \(f \in\Cont(\cmp{X'})\)
  with \(f(\varrho(y)) = 1\)
  and \(f(\varrho(x_\infty)) = 0\).
  Then
  \(\lim f(\varrho(x_\alpha)) = f(\varrho(x_\infty))\neq
  f(\varrho(y))\).
  So~\((x_\alpha)\)
  cannot accumulate at~\(y\).
  Thus~\(x_\infty\)
  is the only possible accumulation point of~\((x_\alpha)\).
  Since~\(\cmp{X}\)
  is compact, this implies that \(\lim x_\alpha = x_\infty\).

  Finally, we prove~\ref{en:convergence3}.  We construct
  \(X'\subseteq \cmp{X'}\)
  as in~\ref{en:convergence1}.  There is a sequence~\((x'_n)\)
  in~\(X'\)
  with \(\lim x'_n = \varrho(x_\infty)\)
  because~\(\cmp{X'}\)
  is second countable.  Since \(X\to X'\)
  is surjective, we may lift it to a sequence~\((x_n)\)
  in~\(X\)
  with \(\lim \varrho(x_n) = \varrho(x_\infty)\).
  Now~\ref{en:convergence2} gives \(\lim x_n = x_\infty\).
\end{proof}

\begin{lemma}
  \label{lem:controlled nbd}
  Let~$X$ be a $\sigma$\nb-compact, locally compact, Hausdorff space
  and let \(X\subseteq \cmp{X}\)
  be a compactification with metrisable boundary~$\bdy{X}$.  The
  corresponding topological coarse structure on~\(X\)
  has a controlled neighbourhood of the diagonal.
\end{lemma}

\begin{proof}
  Construct a second countable compactification
  \(X'\subseteq \cmp{X'}\)
  and a continuous quotient map \(\varrho\colon \cmp{X}\to\cmp{X'}\)
  as in Lemma~\ref{lem:convergence}.\ref{en:convergence1}.  We may
  equip~\(\cmp{X'}\)
  with a metric~\(d'\)
  that defines its topology.  Let $\varrho^*(d')$ be the resulting
  continuous pseudo-metric on~\(\cmp{X}\).
  Since~\(X\)
  is \(\sigma\)\nb-compact,
  there is an increasing family of relatively compact, open subsets
  $U_n\subseteq X$ with \(\bigcup U_n = X\).  For $n\in\N$, let
  \[
  E_n \defeq \{(x, y) \in U_n\times U_n: d(x, y) < 2^{-n}\}.
  \]
  This is open in \(U_n\times U_n\)
  for each~$n$ because~$d$ is continuous on $\cmp{X} \times \cmp{X}$,
  and it contains the diagonal of~\(U_n\).
  Thus the union $E = \bigcup_n E_n$ is an open neighbourhood of the
  diagonal in~\(X\).
  We claim that~\(E\)
  is controlled.  So let \((x_\alpha,y_\alpha)\)
  be a net in~\(E\)
  that converges in \(\cmp{X}\times\cmp{X}\),
  and assume, say, that \(\lim x_\alpha\in \bdy{X}\).
  We must prove \(\lim x_\alpha = \lim y_\alpha\).
  Choose \(n(\alpha)\in\N\)
  minimal with \(x_\alpha\in U_{n(\alpha)}\).
  Then \(\lim n(\alpha)=\infty\)
  because \(\lim x_\alpha\in \bdy{X}\).
  And
  \(d'(\varrho(x_\alpha),\varrho(y_\alpha)) =
  d(x_\alpha,y_\alpha)<2^{-n(\alpha)}\)
  because \((x_\alpha,y_\alpha)\in E\).
  Thus \(\lim \varrho(x_\alpha) = \lim \varrho(y_\alpha)\)
  because~\(d'\)
  defines the topology on~\(\cmp{X'}\).
  Now Lemma~\ref{lem:convergence}.\ref{en:convergence2} gives \(\lim
  x_\alpha = \lim y_\alpha\).
\end{proof}


\bib{Arveson:Extensions_Cstar}{article}{
  author={Arveson, William},
  title={Notes on extensions of $C^*$\nobreakdash -algebras},
  journal={Duke Math. J.},
  volume={44},
  date={1977},
  number={2},
  pages={329--355},
  issn={0012-7094},
  review={\MRref {0438137}{55\,\#11056}},
  doi={10.1215/S0012-7094-77-04414-3},
}

\bib{Banerjee:Thesis}{thesis}{
  author={Banerjee, Tathagata},
  type={phdthesis},
  title={Coarse Geometry for Noncommutative Spaces},
  institution={Georg-August-Universit\"at G\"ottingen},
  date={2015},
}

\bib{Bieliavsky-Gayral:Deformation_Kaehler}{article}{
  author={Bieliavsky, Pierre},
  author={Gayral, Victor},
  title={Deformation quantization for actions of K\"ahlerian Lie groups},
  journal={Mem. Amer. Math. Soc.},
  volume={236},
  date={2015},
  number={1115},
  pages={vi+154},
  issn={0065-9266},
  isbn={978-1-4704-1491-7},
  review={\MRref {3379676}{}},
  doi={10.1090/memo/1115},
}

\bib{Brown:Continuity}{article}{
  author={Brown, Lawrence G.},
  title={Continuity of actions of groups and semigroups on Banach spaces},
  journal={J. London Math. Soc. (2)},
  volume={62},
  date={2000},
  number={1},
  pages={107--116},
  issn={0024-6107},
  review={\MRref {1771854}{2003b:22014}},
  doi={10.1112/S0024610700001058},
}

\bib{Choi-Effros:CP_lifting}{article}{
  author={Choi, Man Duen},
  author={Effros, Edward G.},
  title={The completely positive lifting problem for \(C^*\)\nobreakdash -algebras},
  journal={Ann. of Math. (2)},
  volume={104},
  date={1976},
  number={3},
  pages={585--609},
  issn={0003-486X},
  review={\MRref {0417795}{54\,\#5843}},
  doi={10.2307/1970968},
}

\bib{Coburn-Xia:Toeplitz_Rieffel}{article}{
  author={Coburn, Lewis A.},
  author={Xia, Jingbo},
  title={Toeplitz algebras and Rieffel deformations},
  journal={Comm. Math. Phys.},
  volume={168},
  date={1995},
  number={1},
  pages={23--38},
  issn={0010-3616},
  review={\MRref {1324389}{96b:46101}},
  eprint={http://projecteuclid.org/euclid.cmp/1104272284},
}

\bib{Dixmier:Cstar-algebras}{book}{
  author={Dixmier, Jacques},
  title={\(C^*\)\nobreakdash -Algebras},
  note={Translated from the French by Francis Jellett; North-Holland Mathematical Library, Vol. 15},
  publisher={North-Holland Publishing Co.},
  place={Amsterdam},
  date={1977},
  pages={xiii+492},
  isbn={0-7204-0762-1},
  review={\MRref {0458185}{56\,\#16388}},
}

\bib{Echterhoff-Kaliszewski-Quigg-Raeburn:Categorical}{article}{
  author={Echterhoff, Siegfried},
  author={Kaliszewski, Steven P.},
  author={Quigg, John},
  author={Raeburn, Iain},
  title={A categorical approach to imprimitivity theorems for $C^*$\nobreakdash -dynamical systems},
  journal={Mem. Amer. Math. Soc.},
  volume={180},
  date={2006},
  number={850},
  pages={viii+169},
  issn={0065-9266},
  review={\MRref {2203930}{2007m:46107}},
  doi={10.1090/memo/0850},
}

\bib{Imai-Takai:Duality}{article}{
  author={Imai, Sh\=o},
  author={Takai, Hiroshi},
  title={On a duality for $C^*$\nobreakdash -crossed products by a locally compact group},
  journal={J. Math. Soc. Japan},
  volume={30},
  date={1978},
  number={3},
  pages={495--504},
  review={\MRref {500719}{}},
  doi={10.2969/jmsj/03030495},
}

\bib{Kaschek-Neumaier-Waldmann:Positivity}{article}{
  author={Kaschek, Daniel},
  author={Neumaier, Nikolai},
  author={Waldmann, Stefan},
  title={Complete positivity of Rieffel's deformation quantization by actions of~$\mathbb {R}^d$},
  journal={J. Noncommut. Geom.},
  volume={3},
  date={2009},
  number={3},
  pages={361--375},
  issn={1661-6952},
  review={\MRref {2511634}{2011a:53178}},
  doi={10.4171/JNCG/40},
}

\bib{Kasprzak:Rieffel_deformation}{article}{
  author={Kasprzak, Pawe\l },
  title={Rieffel deformation via crossed products},
  journal={J. Funct. Anal.},
  volume={257},
  date={2009},
  number={5},
  pages={1288--1332},
  issn={0022-1236},
  review={\MRref {2541270}{2010e:46073}},
  doi={10.1016/j.jfa.2009.05.013},
}

\bib{Lance:Hilbert_modules}{book}{
  author={Lance, E. {Ch}ristopher},
  title={Hilbert $C^*$\nobreakdash -modules},
  series={London Mathematical Society Lecture Note Series},
  volume={210},
  publisher={Cambridge University Press},
  place={Cambridge},
  date={1995},
  pages={x+130},
  isbn={0-521-47910-X},
  review={\MRref {1325694}{96k:46100}},
  doi={10.1017/CBO9780511526206},
}

\bib{Meyer:Generalized_Fixed}{article}{
  author={Meyer, Ralf},
  title={Generalized fixed point algebras and square-integrable groups actions},
  journal={J. Funct. Anal.},
  volume={186},
  date={2001},
  number={1},
  pages={167--195},
  issn={0022-1236},
  review={\MRref {1863296}{2002j:46086}},
  doi={10.1006/jfan.2001.3795},
}

\bib{Neshveyev:Smooth}{article}{
  author={Neshveyev, Sergey V.},
  title={Smooth crossed products of Rieffel's deformations},
  journal={Lett. Math. Phys.},
  volume={104},
  date={2014},
  number={3},
  pages={361--371},
  issn={0377-9017},
  review={\MRref {3164614}{}},
  doi={10.1007/s11005-013-0675-9},
}

\bib{Paulsen:Completely_bounded}{book}{
  author={Paulsen, Vern},
  title={Completely bounded maps and operator algebras},
  series={Cambridge Studies in Advanced Mathematics},
  volume={78},
  publisher={Cambridge University Press, Cambridge},
  date={2002},
  pages={xii+300},
  isbn={0-521-81669-6},
  review={\MRref {1976867}{2004c:46118}},
  doi={10.1017/CBO9780511546631},
}

\bib{Rieffel:Proper}{article}{
  author={Rieffel, Marc A.},
  title={Proper actions of groups on $C^*$\nobreakdash -algebras},
  conference={ title={Mappings of operator algebras}, address={Philadelphia, PA}, date={1988}, },
  book={ series={Progr. Math.}, volume={84}, publisher={Birkh\"auser Boston}, place={Boston, MA}, },
  date={1990},
  pages={141--182},
  review={\MRref {1103376}{92i:46079}},
  doi={10.1007/978-1-4612-0453-4_6},
}

\bib{Rieffel:Metrics_state}{article}{
  author={Rieffel, Marc A.},
  title={Metrics on state spaces},
  journal={Doc. Math.},
  volume={4},
  date={1999},
  pages={559--600},
  issn={1431-0635},
  review={\MRref {1727499}{2001g:46154}},
  eprint={http://www.math.uni-bielefeld.de/documenta/vol-04/17.html},
}

\bib{Rieffel:Integrable_proper}{article}{
  author={Rieffel, Marc A.},
  title={Integrable and proper actions on $C^*$\nobreakdash -algebras, and square-integrable representations of groups},
  journal={Expo. Math.},
  volume={22},
  date={2004},
  number={1},
  pages={1--53},
  issn={0723-0869},
  review={\MRref {2166968}{2006g:46108}},
  doi={10.1016/S0723-0869(04)80002-1},
}

\bib{Roe:Lectures}{book}{
  author={Roe, John},
  title={Lectures on coarse geometry},
  series={University Lecture Series},
  volume={31},
  publisher={Amer. Math. Soc.},
  place={Providence, RI},
  date={2003},
  pages={viii+175},
  isbn={0-8218-3332-4},
  review={\MRref {2007488}{2004g:53050}},
  doi={10.1090/ulect/031},
}

\bib{Roe:Correction_Lectures}{book}{
  author={Roe, John},
  title={Corrections to ``Lectures on coarse geometry''},
  date={2005-08-11},
  eprint={http://www.personal.psu.edu/jxr57/writings/correction.pdf},
}

\bib{Wegge-Olsen:K-theory}{book}{
  author={Wegge-Olsen, Niels Erik},
  title={\(K\)\nobreakdash -Theory and \(C^*\)\nobreakdash -algebras},
  series={Oxford Science Publications},
  publisher={The Clarendon Press Oxford University Press},
  place={New York},
  date={1993},
  pages={xii+370},
  isbn={0-19-859694-4},
  review={\MRref {1222415}{95c:46116}},
}

\bib{Willett:Index_band-dominated}{article}{
  author={Willett, Rufus},
  title={An index theorem for band-dominated operators with slowly oscillating coefficients (after Deundyak and Shteinberg)},
  journal={Integral Equations Operator Theory},
  volume={69},
  date={2011},
  number={3},
  pages={301--316},
  issn={0378-620X},
  review={\MRref {2774605}{}},
  doi={10.1007/s00020-010-1857-9},
}



\begin{bibdiv}
  \begin{biblist}
    \bibselect{references}
  \end{biblist}
\end{bibdiv}
\end{document}